\documentclass[pdflatex,12pt,a4paper,twoside]{article}
\usepackage{authblk}
\usepackage{amsmath,amsfonts,amssymb,amsthm}
\usepackage{ifthen}
\usepackage{graphicx}
\usepackage{epsfig}
\usepackage{nicefrac}
\usepackage{mathrsfs}
\usepackage[a4paper,left=32mm,right=32mm,top=36mm,bottom=35mm]{geometry}
\usepackage{amsmath}
\usepackage{amsfonts}
\newcommand{\R}{\mathbb{R}}

\newcommand{\Z}{\mathbb{Z}}
\newcommand{\eps}{\varepsilon}
\newcommand{\fhi}{\varphi}

\newcommand{\weak}{\rightharpoonup}
\newcommand{\weakto}{\rightharpoonup}

\newcommand{\mean}{-\hspace{-1.15em}\int}

\newcommand{\del}{\partial}

\newcommand{\curl}{\mathrm{curl}}

\newcommand{\diam}{\mathrm{diam}}
\newcommand{\eff}{\mathrm{\tiny eff}}

\def\calM{\mathcal{M}}

\def\calH{\mathcal{H}}

\def\XXint#1#2#3{{\setbox0=\hbox{$#1{#2#3}{\int}$}
     \vcenter{\hbox{$#2#3$}}\kern-.5\wd0}}

\usepackage[
    bookmarks,
    colorlinks=true,
    bookmarksopen=true,
    bookmarksopenlevel=1,
    linkcolor=black,
    citecolor=blue,
    urlcolor=blue,
    filecolor=blue,
    anchorcolor=red,
    pdfstartview=FitH,
    pdfpagelayout=OneColumn,
    plainpages=false, 
    pdfpagelabels, 
    hypertexnames=false, 
    linktocpage, 
]{hyperref}

\newtheorem{theorem}{Theorem}[section]
\newtheorem{definition}[theorem]{Definition}
\newtheorem{lemma}[theorem]{Lemma}
\newtheorem{proposition}[theorem]{Proposition}

\newtheorem{remark}[theorem]{Remark}

\numberwithin{equation}{section}

\title{\bf\Large Effective acoustic properties of a meta-material consisting of
  small Helmholtz resonators}

\author{A.\,Lamacz, B.\,Schweizer\thanks{Technische Universit\"at
    Dortmund, Fakult\"at f\"ur Mathematik, Vogelpothsweg 87, D-44227
    Dortmund, Germany.}}

\begin{document}

\maketitle

\begin{abstract}
  We investigate the acoustic properties of meta-materials that are
  inspired by sound-absorbing structures. We show that it is possible
  to construct meta-materials with frequency-dependent effective
  properties, with large and/or negative
  permittivities. Mathematically, we investigate solutions $u^\eps:
  \Omega_\eps \to \R$ to a Helmholtz equation in the limit $\eps\to 0$
  with the help of two-scale convergence. The domain $\Omega_\eps$ is
  obtained by removing from an open set $\Omega\subset \R^n$ in a
  periodic fashion a large number (order $\eps^{-n}$) of small
  resonators (order $\eps$). The special properties of the
  meta-material are obtained through sub-scale structures in the
  perforations.
\end{abstract}



\pagestyle{myheadings} 
\thispagestyle{plain} 

\markboth{A.\,Lamacz, B.\,Schweizer}{Effective Helmholtz Equation in a
  geometry with small resonators}

\section{Introduction}

In this article, we are interested in the acoustic properties of a
particular meta-material, inspired by sound absorbing structures.  We
define a complex geometry, consisting of many small cavities, and
study the Helmholtz equation in this geometry. The acoustic properties
of the meta-material are determined by the Helmholtz equation since
the acoustic pressure $p$ of a time-harmonic sound wave of fixed
frequency $\omega$ is of the form $p(x,t) = u(x) e^{i\omega t}$, where
$u$ solves a Helmholtz equation.

In standard homogenization settings, nothing special can be expected
concerning the acoustic properties of a meta-material (e.g.\,large or
negative coefficients). Instead, in this contribution, we introduce a
setting where the small inclusions are resonators and where the
effective behavior of the meta-material introduces new features.

Let us describe these statements in a more mathematical language: We
consider a domain $\Omega_\eps\subset \R^n$, $n=2$ or $n=3$, which is
obtained by removing small obstacles of typical size $\eps>0$ from a
domain $\Omega\subset \R^n$. For a fixed frequency $\omega\in\R$, we
study solutions $u^\eps\in H^1(\Omega_\eps)$ to the Helmholtz equation
\begin{align}
  \begin{split}
    \label{eq:Helmholtzeps}
    -\Delta u^\eps&=\omega^2u^\eps\quad\,\text{ in }\Omega_\eps\,,\\
    \partial_n u^\eps&=0\quad\quad\ \ \text{ on }
    \partial\Omega_\eps\setminus\partial\Omega\,,\\
    u^\eps&=g\quad\quad\ \ \text{ on }\partial\Omega\,.
  \end{split}                        
\end{align}
The first boundary condition expresses that the obstacles are
sound-hard (homogeneous Neumann condition, $n$ denotes the exterior
normal at $\del\Omega_\eps$), the second boundary condition prescribes
a pressure at the external boundary, $g\in H^1(\Omega)$ is responsible
for the generation of a sound wave in the domain.

When $\Omega_\eps$ is obtained from $\Omega$ by a standard periodic
perforation procedure, then the homogenization of equation \eqref
{eq:Helmholtzeps} is well-established.  One finds an effective
coefficient $A_*$ and a volume correction factor $\lambda\in
(0,\infty)$ such that, for small $\eps>0$, the solution $u^\eps$ looks
essentially like the solution $u^*$ of the effective equation
$-\nabla\cdot (A_* \nabla u^*) = \omega^2 \lambda u^*$. In this
effective system, neither $A^*$ nor $\lambda$ are frequency dependent.

In contrast to such a standard approach we investigate \eqref
{eq:Helmholtzeps} for a domain $\Omega_\eps$, where every single
inclusion (perforation) has the shape of a small resonator. This is
possible by introducing a three-scale problem: The macro-scale is
$\diam(\Omega) = O(1)$, the micro-scale of the single inclusion is
$O(\eps)$, and the single inclusion contains a subscale feature of
either size $O(\eps^2)$ (in dimension $n=3$) or $O(\eps^3)$ (in
dimension $n=2$). In this three-scale domain $\Omega_\eps$, the
solutions $u^\eps$ exhibit a more interesting behavior. We perform the
homogenization procedure and find that, for small $\eps>0$, the
solution $u^\eps$ to \eqref {eq:Helmholtzeps} looks essentially like
the solution $v$ to the effective system
\begin{equation}
  \label{eq:effectivesystem-prolog}
  -\nabla\cdot\left(A_{*}\nabla v\right) 
  =\omega^2\Lambda\, v\quad\text{ in }\Omega\,.
\end{equation}
The form of this system is as in the standard homogenization setting,
two effective coefficients $A_{*}$ and $\Lambda\in \R$ modify the
original equation when describing the system with a macroscopic
equation on $\Omega$. But, due to the more complex geometry, we obtain a
parameter $\Lambda = \Lambda(\omega)$, which is frequency dependent.
It can change sign and it can be arbitrarily large due to a resonance
effect in the single cavity. The resonance frequency is given by the
well-known formula for Helmholtz resonators, $\omega_* =
\sqrt{A/(LV)}$, where the real numbers $A$, $L$, and $V$ characterize
the geometric properties of the resonators (area of a channel cross
section, length of the channel, volume of the resonator).

Usually, small inclusions correspond to a high resonance frequency,
and not to some finite frequency $\omega_*$. But, as was shown in
\cite{Schweizer-HelmRes}, a finite resonance frequency can be obtained
when a singular structure is included in the geometry: We consider a
setting where, in every periodicity cell $Y\subset \R^n$, a resonator
region $R_Y\subset Y$ is separated from an exterior region $Q_Y\subset
Y$ by the sound-hard obstacle $\Sigma_Y^\eps\subset Y$, cp. Figure
\ref {fig:manyhelm}. But the separation is not complete, the obstacle
leaves open a small channel that connects $R_Y$ and $Q_Y$. The scaling
of the channel depends on the dimension.  In two space dimensions ($n
= 2$), the relative scaling of the opening is $\eps^p$ with $p=2$,
hence the channel width of the single inclusion $\Sigma_k^\eps\subset
\Omega$ (where $k\in \Z^n$ is the index of the $k$-th inclusion) is of
order $\eps^3$. In dimension $n=3$ the exponent is $p=1$, the channel
opening diameter of the single inclusion $\Sigma_k^\eps\subset \Omega$
is therefore of order $\eps^2$. In both cases, the scaling is such
that the quantity $A_\eps/(L_\eps V_\eps)$ is of order $1$, where
$A_\eps$ is the channel opening, $L_\eps$ is the channel length and
$V_\eps$ is the volume. Let us check this condition: For $n=2$ we have
$A_\eps/(L_\eps V_\eps) \sim \eps^{p+1} / (\eps^1 \eps^n) = 1$, for
$n=3$ we have $A_\eps/(L_\eps V_\eps) \sim (\eps^{p+1})^2 / (\eps^1
\eps^n) = 1$.

\subsection{Main result}

We investigate a large domain $\Omega\subset \R^n$ that contains
meta-material in some region $D\subset \Omega$. The single small
resonator is denoted as $\Sigma_k^\eps$, with $k\in \Z^n$ such that
$\eps\left(k+Y\right)\subset D$, where $Y:=(-\frac12,\frac12)^n$. 
The union of all resonators $\Sigma_\eps = \bigcup_k
\Sigma_k^\eps\subset D$ defines the perforated domain $\Omega_\eps :=
\Omega\setminus \Sigma_\eps$. In order to analyze the effect of the
resonator region $D$, we study solutions $u^\eps$ to the Helmholtz
equation \eqref {eq:Helmholtzeps} and investigate their behavior
inside and outside of $D$ in the limit $\eps\to 0$.

\begin{figure}[th]
   \centering
   \includegraphics[height=65mm]{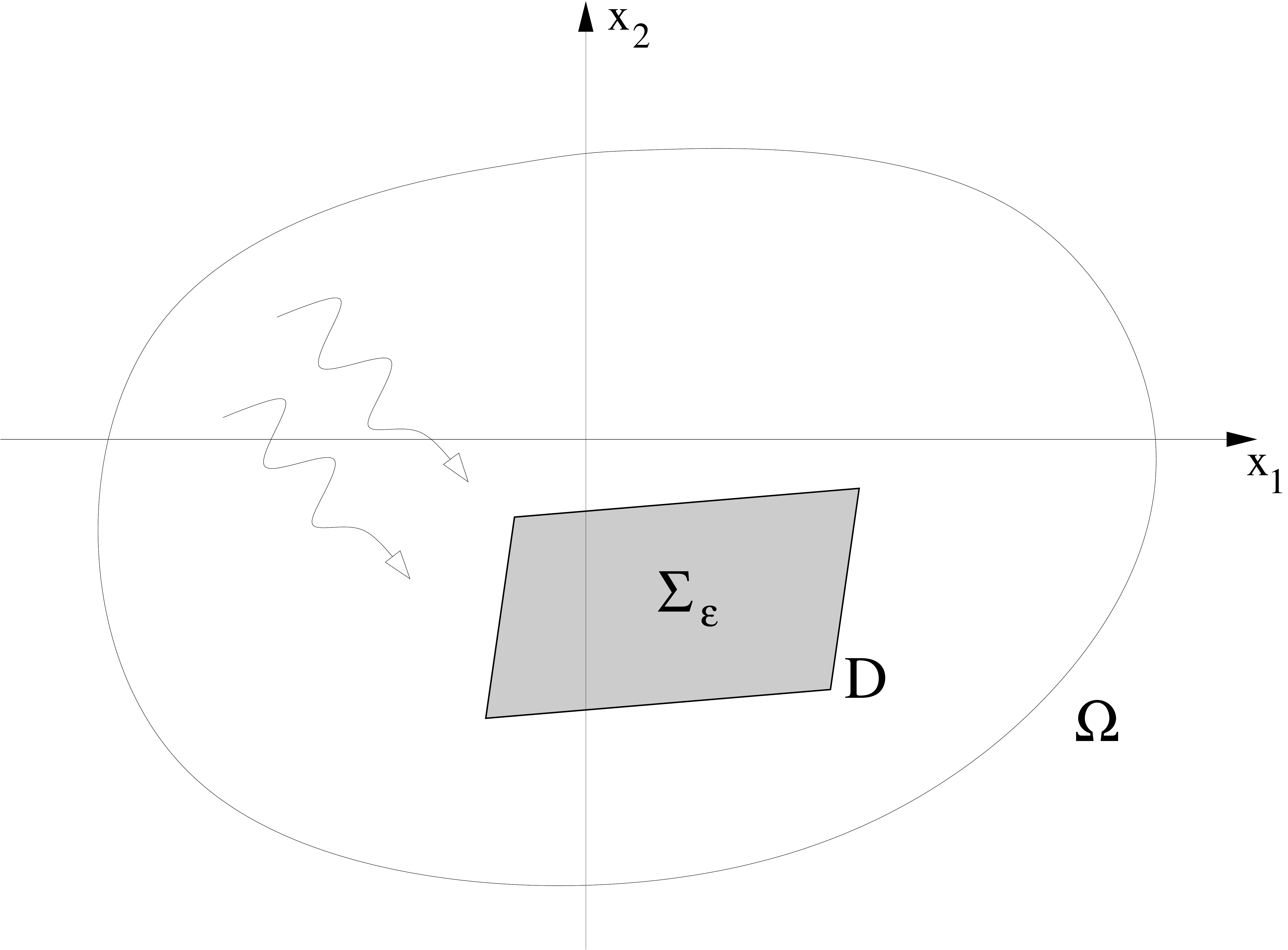}\hspace*{8mm}
   \includegraphics[height=38mm]{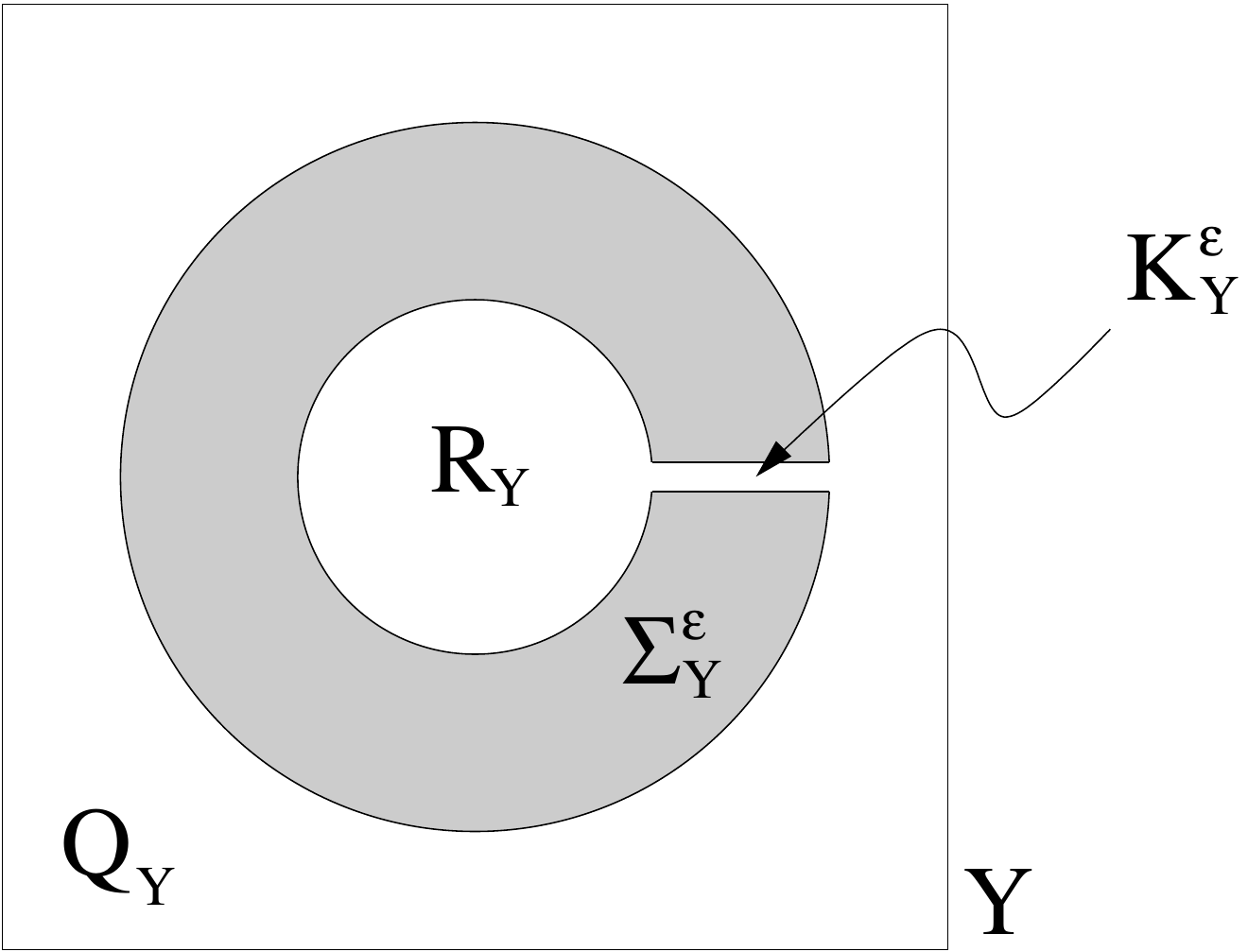}
   \caption{\em Sketch of the scattering problem. 
     Left: The sub-region $D\subset \Omega$ contains the small 
     Helmholtz resonators, given by $\Sigma_\eps\subset D$. The number 
     of resonators in the region $D$ is of order $\eps^{-n}$. 
     We are interested in the effective properties of the meta-material in $D$.
     Right: The microscopic geometry with the single resonator $R_Y$. 
     The channel width inside $Y$ is of the order $\eps^p$.
     \label{fig:manyhelm}}
\end{figure}

We derive an effective Helmholtz equation with the tool of two-scale
convergence. Essentially, the effective system is given by \eqref
{eq:effectivesystem-prolog}.  In this equation, the effective
permittivity $\Lambda:\Omega\to \R$ is $\Lambda(x)=1$ for
$x\in\Omega\setminus D$ (outside the region that contains the
resonators) and $\Lambda(x)=\Lambda_{\eff}$ for $x\in D$,
where the real number
\begin{align}
  \label{eq:Lambda-1}
  \Lambda_{\eff} := Q - \frac{A}{L} 
  \left(\omega^2-\frac{A}{LV}\right)^{-1}
\end{align}
is determined by the positive real numbers $A$, $L$, $V$ and $Q$ which
characterize the geometric properties of the resonators (area of a
channel cross section, length of the channel, volume of the resonator,
volume of the exterior), cf. Section \ref{sec.geometry} below. The
number $\Lambda_{\eff}$ represents the permittivity of the effective
medium. Due to resonance properties of $\Sigma_Y^\eps$, it can be
negative and it can be large in absolute value (with both signs). The
resonance frequency $\omega_* = \sqrt{A/(LV)}$ is determined by the
geometry.

The ellipticity matrix is $A_{*}(x) = \text{{\bf 1}}_{\R^n}$ for
$x\in\Omega\setminus D$, whereas for $x\in D$ it is given as a cell
problem integral:
\begin{align}
  \label{eq:effectiveA}
  \left( A_{*}\right)_{ij}(x) :=
  \left( A_{\eff}\right)_{ij} &:=
  \int_{Q_Y}\left[\delta_{ij}+\partial_{y_i}\chi_j(y)\right]\,dy\quad\text{
    for }i,j\in\{1,...,n\}\,,
\end{align}
where $\delta_{ij}$ denotes the Kronecker Delta and $\chi_j$ is the
solution to the cell problem \eqref{eq:cellproblemw}. The set
$Q_Y\subset Y$ is the part of the periodicity cell $Y$ that is
exterior to the obstacle and $Q:=|Q_Y|$ is its volume.

Let us formulate here our main result in a condensed form. Theorem
\ref {thm:main-1} characterizes the effective influence of the
scattering region $D$ on waves in $\Omega$. Theorem \ref {thm:main-1}
is a consequence of the stronger result in Theorem \ref {thm:main-2},
which includes a characterization of $u^\eps$ in $D$.

\begin{theorem}[Effective Helmholtz equation]
  \label{thm:main-1}
  Let $\Omega\subset \R^n$ be a domain and let $D\subset \Omega$
  contain the obstacle set $\Sigma_\eps\subset D$ as described in
  Section \ref {sec.geometry}. Let $u^\eps\in H^1(\Omega_\eps)$ be a
  sequence of solutions to \eqref{eq:Helmholtzeps} satisfying
  \begin{align}
    \label{eq:uniformboundassume}
    \|u^\eps\|_{L^2(\Omega_\eps)}\leq C\,.
  \end{align}
  Let $\Lambda_{\eff}$ and $\Lambda$ be defined by the
  algebraic relation \eqref {eq:Lambda-1}, which uses the geometric
  constants $A, L, V, Q>0$. Let $A_{*}$ be defined by
  \eqref {eq:effectiveA} with the help of
  cell problems. Let $v\in H^1(\Omega)$ be a solution to the effective
  Helmholtz equation
  \begin{align}
    \label{eq:effectiveHelmholtz-1}
    -\nabla\cdot\left(A_{*}\nabla v\right) 
    &=\omega^2\Lambda\, v\quad\text{ in }\Omega\,,\\
    \label{eq:effectiveHelmholtz-2}
    v&=g\quad\quad\quad\text{ on }\partial\Omega\,.
  \end{align}
  If the solution $v$ to the effective system is unique, then there
  holds 
  \begin{equation}
    u^\eps |_{\Omega\setminus D} \to v |_{\Omega\setminus D}\quad \text{ in }
    L^2(\Omega\setminus D)\,.
  \end{equation}
\end{theorem}

\begin{remark}
 For almost every $\omega\in\R$ the solution to  \eqref{eq:effectiveHelmholtz-1}-\eqref{eq:effectiveHelmholtz-2} 
is unique, since the generalized Eigenvalue-problem \eqref{eq:effectiveHelmholtz-1} has a discrete set of Eigenvalues $\omega^2$.  
\end{remark}

The result of Theorem \ref{thm:main-1} yields that the field $u^\eps$ is determined outside the resonator
region $D$ by the effective system
\eqref{eq:effectiveHelmholtz-1}--\eqref{eq:effectiveHelmholtz-2}.
Inside $D$, the solutions $u^\eps$ are oscillatory due to the
micro-structure. We describe the two-scale limit of $u^\eps$ in $D$ in
Subsection \ref {ssec.twoscale-structure}.

\subsection{Literature}

Homogenization theory is concerned with the derivation of effective
equations. The most elementary task is to consider a sequence $u^\eps$
of solutions to a family of partial differential equations that
contains a small parameter $\eps$, very often the periodicity of the
geometry or the periodicity of a coefficient. The goal is to find an
equation that characterizes the weak limit $u^*$ of the sequence
$u^\eps$. Our result is of that form. The general theory has origins
in \cite {Sanchez-Palencia} and it was greatly simplified by
introducing the notion of two-scale convergence \cite {Allaire1992,
  Nguetseng1989}. The theory was later extended in several directions,
e.g.\,to stochastic problems \cite {ZhikovMR1329546}, to problems with
multiple scales \cite {AllaireBriane}, and to problems with measure
valued limits \cite{BouchitteFelbacq-2000, CherednichenkoSZ}.

{\bf Perforated media.} Our result treats perforated media in the
sense that an equation with constant coefficients is considered on a
domain $\Omega_\eps$ which is obtained by removing many small
subdomains from a macroscopic domain $\Omega$. A simple boundary
condition is imposed on the boundary $\del\Omega_\eps\setminus
\del\Omega$. Perforated media have been studied with Dirichlet
boundary conditions in \cite{CioranescuMurat-strange} and with Neumann
boundary conditions in \cite {CioranescuPaulin1979}, the scattering
problem has been analyzed in \cite {CodegoneMR2126242}. For
perforations along a lower dimensional manifold see e.g.\,\cite
{Doerlemann-Heida-Schweizer}. Periodically perforated domains can
also be treated with the tool of two-scale convergence, which shortens
the proofs of some of the original results.  The difference between
our result and those mentioned above is that we introduce a sub-scale
structure in the perforation to allow for resonances. This leads to
more interesting terms in the effective equation.

{\bf Resonances in Maxwell's equations.} Resonances in the single
periodicity cell can create interesting effects in the homogenization limit. 
One of the fascinating examples is the construction of negative index
materials for light as suggested in \cite {BrienPendry2002,
  Pendry2000} with its possible applications to cloaking (see \cite
{Kohn-LSW} and the references therein). The meta-material for the
Maxwell's equations has qualitatively new features, namely a negative
effective permittivity and a negative effective permeability. 
The astonishing properties of such a material had been anticipated in
\cite {Veselago1968}, but the mathematical analysis of negative index 
meta-materials is much younger. Model equations have been studied in
\cite {KohnShipman}, non-rigorous results appeared in \cite
{ChenLipton2010, GuenneauZollaNicolet}. The mathematical theory that
confirmed {\em negative effective permeability} appeared in \cite
{BouchitteSchweizer-Max, Lamacz-Schweizer-Max} for metals and in \cite
{BouchitteBourel2009} for dielectrics. The full result on the {\em
  negative effective index} meta-material can be found in \cite
{Lamacz-Schweizer-Neg}.

{\bf Generation of resonances.} At first sight, a resonance effect in
a structure of order $\eps$ seems impossible when the frequency
$\omega$ is kept fixed: One expects that an object of order $\eps$ has
a resonance at wave-length of order $\eps$ and hence a resonance
frequency of order $\eps^{-1}$.  Indeed, in order to obtain
interesting features in the limit $\eps\to 0$, one has to introduce
some singular behavior of the micro-structure.

{\em (a) Large contrast.}  The simplest setting uses a large contrast,
e.g.\,a parameter $a_\eps$ that is of order $1$ outside the structure
and $\eps^2$ in the structure \cite {BouchitteBourel2009,
  BouchitteSchweizer-Plasmons, ChenLipton2010,
  Lamacz-Schweizer-Neg}. In many cases, the large contrast must be
combined with one of the features (b)-(d) below.

{\em (b) Fibers.}  Interesting effects within the framework of
homogenization can be obtained when fibers are present in the
structure \cite {BellieudBouchitte1998, BellieudGruias2005,
  BellieudBouchitte2002, ChenLipton2010}. On the one hand a fiber can be regarded as
a classical periodic micro-structure ($\eps$-periodic in each direction). 
On the other hand, the exterior of the
fiber (in the single cell) is not a simply connected set, which makes
some standard two-scale limit characterizations impossible. Other
effects are related to the fact that the length of a single fiber is of order $O(1)$. 
For implications in Maxwell equations see e.g.\,\cite
{BouchitteFelbacq2004, BouchitteFelbacq2006, ChenLipton2010,
  Lamacz-Schweizer-Neg}.

{\em (c) Split rings.} More complicated is the three-dimensional
construction of a split ring: For $\eps>0$, the ring is simply
connected, but the slit closes in the limit $\eps\to 0$ such that the
limiting object is no longer simply connected. This change of topology
is exploited in \cite {BouchitteSchweizer-Max, Lamacz-Schweizer-Max}.

{\em (d) Disconnected subregions.}  In the work at hand we use a
different construction, based on the Helmholtz resonator that has been
studied mathematically in \cite {Schweizer-HelmRes}. Again, a
topological effect plays a role: Every cell has an interior part $R_Y$
(the resonator) and an exterior part $Q_Y$. For every $\eps>0$, the
two subdomains are connected by a thin channel, but since the channel
gets thinner and thinner, the limiting object for $\eps\to 0$ is no
longer connected, it has the two disconnected components $R_Y$ and
$Q_Y$.

In both settings (c) and (d) it is the change of the topology that
makes the resonance in the small object at finite frequency
possible. We note that the analysis of spectral properties in
perforated domains require other methods if the low- and
high-frequency parts of the spectrum are studied separately, see \cite
{Allaire-Conca-1998, ZhikovMR1329546, Zhikov-spectral}.


\subsection{Geometry}
\label{sec.geometry}

Let $\Omega\subset\R^n$ be open and bounded and let $D\subset \Omega$
be an open set with Lipschitz boundary such that $\bar D\subset
\Omega$. The set $D$ contains the periodic perforations, it is the
scatterer in the effective equations.

\smallskip {\bf Microscopic geometry.} We start the construction from
the periodicity cube $Y:=(-\frac12,\frac12)^n$. Since we will always
impose periodicity conditions on the cube $Y$, we may identify it with
the torus $\mathbb{T}^{n}$. We assume that $Y$ is the disjoint union
\begin{equation*}
  Y = R_Y \cup \bar\Sigma_Y \cup Q_Y\,, 
\end{equation*}
each of the three sets $R_Y, \Sigma_Y, Q_Y$ is open and connected with
Lipschitz boundary, $\bar R_Y, \bar \Sigma_Y\subset
(-\frac12,\frac12)^n$ do not touch the boundary, and $R_Y\cup Q_Y$ is
{\em not} connected, see Figure \ref{fig:manyhelm}.

{\it Channel construction.} The channel is constructed starting from a
one-dimensional line-segment $\Gamma_Y\subset\Sigma_Y$ that connects
$\del R_Y$ with $\del Q_Y$.  We may write the segment with its
tangential vector $\tau_{\Gamma_Y}\in \R^n$ as $\Gamma_Y = \{\gamma_0
+ t\tau_{\Gamma_Y} : t \in (0, L) \}$ with $L$ denoting the length of
the segment. For ease of notation we assume in the following that the
line segment has the tangential vector $\tau_{\Gamma_Y} = e_1$ and
that $\Gamma_Y \subset \R\times \{ 0\} \subset \R^n$. In this case, we
have $\gamma_0 = (y_R,0)$ and $\Gamma_Y = (y_R,y_Q)\times \{0\}$ with
$y_Q := y_R + L < 1/2$.

For technical reasons in the study of the asymptotic behavior of cell
solutions, we assume that the boundaries $\del R_Y$ and $\del Q_Y$ are
flat in the vicinity of $(y_R,0)$ and $(y_Q,0)$: For some $\delta>0$
there holds $\del R_Y\cap B_\delta((y_R,0)) \subset \{y_R\}\times
\R^{n-1}$ and similarly for $\del Q_Y$.

We now construct the channel as $K_Y^\eps :=
B_{\alpha\eps^p}(\Gamma_Y) \cap \Sigma_Y$, where $\alpha>0$ is fixed and $B_{\alpha\eps^p}$
denotes the generalized ball around a set. In our simplified setting,
the channel is the cylinder $K_Y^\eps = (y_R,y_Q)\times
B^{n-1}_{\alpha\eps^p}(0) \subset Y$, where
$B^{n-1}_{\alpha\eps^p}(0)$ is the $n-1$-dimensional ball with radius
${\alpha\eps^p}$.  The Helmholtz resonator $\Sigma_Y^\eps$ is defined
as
\begin{align}
  \Sigma_Y^\eps := \Sigma_Y\setminus \bar K_Y^\eps =
  Y \setminus \left(\bar R_Y\cup \bar K_Y^\eps\cup \bar Q_Y\right).
\end{align}
In the following, three geometric quantities will be crucial: (i) The
length $L$ of the channel. (ii) The relative cross section area $A$ of
the channel, $A = 2\alpha$ for $n=2$ and $A = \alpha^2 \pi$ for $n=3$.
(iii) The volume $V=|R_Y|$ of the inner connected component $R_Y$. We
will see that in the effective equation resonances appear at
frequencies $\omega^2$ that are close to the ratio $A/(LV)$.

\smallskip {\bf Macroscopic geometry.} In order to define the domain
$\Omega_\eps$, we use indices $k\in\Z^n$ and shifted small cubes
$Y_k^\eps:=\eps(k+Y)$.  We denote by $\mathcal{K}:=\{k\in
\Z^n\,|\,Y_k^\eps\subset D\}$ the set of indices $k$ such that the
small cube $Y_k^\eps$ is contained in $D$. Here and in the following,
in summations or unions over $k$, the index $k$ takes all values in
the index set $\mathcal{K}$.  The number of relevant indices is of the
order $|\mathcal{K}|=O(\eps^{-n})$.

Using the local subset $\Sigma_Y^\eps \subset Y$ we define the union
$\Sigma_\eps$ of scaled obstacles and the perforated domain
$\Omega_\eps$ by
\begin{align}
  \Sigma_\eps:=\bigcup_{k\in\mathcal{K}}\Sigma_k^\eps:=\bigcup_{k\in\mathcal{K}}\eps(k+\Sigma^\eps_Y)\,,
  \quad\quad
  \Omega_\eps:= \Omega\setminus\bar\Sigma_\eps\,.
\end{align}
Also the other microscopic quantities have their counterpart in the
macroscopic domain: The union $K_\eps$ of the channels, the union
$Q_\eps$ of exterior components and the union $R_\eps$ of interior
components are
\begin{align}
  K_\eps:=\bigcup_{k\in\mathcal{K}}\eps(k+K_Y^\eps)\,,\quad
  Q_\eps:=\bigcup_{k\in\mathcal{K}}\eps(k+Q_Y)\,,\quad
  R_\eps:=\bigcup_{k\in\mathcal{K}}\eps(k+R_Y)\,.
\end{align}

\subsection{Characterization of solutions in $D$}
\label{ssec.twoscale-structure}

Outside the scatterer region $D$, the function $v$ of the effective
system \eqref {eq:effectiveHelmholtz-1}--\eqref
{eq:effectiveHelmholtz-2} is the weak limit of $u^\eps$. In the
following, we want to describe the meaning of $v$ in the scatterer
region $D$.

If we denote by $\tilde{u^\eps}\in L^2(\Omega)$ the trivial extension
of $u^\eps$ by zero, then the uniform bound \eqref
{eq:uniformboundassume} implies the existence of a subsequence and of
a two-scale limit $u_0=u_0(x,y)$ such that $\tilde{u^\eps}
\stackrel{2s}{\weak} u_0$ (for the definition of two-scale convergence
we refer to \cite{Allaire1992}).  Proposition \ref{prop:charactu0}
below yields that the two-scale limit for $x\in D$ is of the form
\begin{align}
  \begin{split}
    \label{eq:charactv}
    u_0(x,y)=
    \begin{cases}
        v(x)& \;\;\text{ for } y\in Q_Y, \\
        w(x)& \;\; \text{ for } y\in R_Y, \\
        0& \;\; \text{ for } y\in \Sigma_Y.
    \end{cases}
  \end{split}
\end{align}
This characterization clarifies the meaning of $v$.  Outside the
perforated region $D\subset\Omega$, there is no micro-structure and
the sequence $u^\eps$ converges strongly to $v$.  Instead, for $x\in
D$, the function $v$ describes the value of $u_0$ in the exterior
$Q_Y$ of the Helmholtz resonator. Outside $D$, the limit function $v$
is comparable to $u^\eps$, while inside $D$ the function $v$ stands
for the values of $u^\eps$ outside the scatterers. This fits with the
result $v\in H^1(\Omega)$ from Lemma \ref{lem:twoscalegradients} below: 
It is the function $v$ that has a weak
continuity property across the boundaries of $D$.

\section{Two-scale limits}

We will always work with the assumption that the sequence $u^\eps$ is
uniformly bounded in $L^2(\Omega_\eps)$ as demanded in \eqref
{eq:uniformboundassume}. We remark that \eqref {eq:uniformboundassume}
implies also a uniform $H^1$-bound,
\begin{align}
  \label{eq:uniformboundimplied}
  \|u^\eps\|_{H^1(\Omega_\eps)}\leq C
\end{align}
for some $\eps$-independent constant $C>0$.  Since the boundary values are given
by $g\in H^1(\Omega)$ and since the domain $\Omega$ is bounded, the
estimate \eqref {eq:uniformboundimplied} follows immediately from
\eqref {eq:uniformboundassume} by testing \eqref {eq:Helmholtzeps}
with the solution $u^\eps$.

In this section we derive Relation \eqref{eq:charactv} for the
two-scale limit $u_0$. Moreover, we provide a characterization for the
two-scale limit of the gradients $\nabla u^\eps$.

\subsection{The two-scale limit $u_0$}
The following proposition provides a first characterization of $u_0$.
We use the sequence $\tilde{u^\eps}\in L^2(\Omega)$ of trivial
extensions of $u^\eps$ (obtained by setting $\tilde{u^\eps}(x) := 0$
for $x\in \Sigma_\eps$).

\begin{proposition}[The two-scale limit $u_0$]
  \label{prop:charactu0}
  Let $u^\eps\in H^1(\Omega_\eps)$ be a sequence of solutions to
  \eqref{eq:Helmholtzeps}.  We assume that the sequence satisfies the
  uniform bound \eqref{eq:uniformboundassume} and that
  $\tilde{u^\eps}\in L^2(\Omega)$ converges in two scales to some
  limit function $u_0\in L^2(\Omega\times Y)$. Then there exist $v\in
  L^2(\Omega)$ and $w\in L^2(D)$ such that
  \begin{enumerate}
  \item For $x\in\Omega\setminus D$ one has $u_0(x,y)=v(x)$.
  \item For $x\in D$ one has
    \begin{align}
      \begin{split}
        \label{eq:charactvnew}
	u_0(x,y)=
        \begin{cases}
            v(x)& \;\;\text{ for } y\in Q_Y, \\
            w(x)& \;\; \text{ for } y\in R_Y, \\
            0& \;\; \text{ for } y\in \Sigma_Y.
        \end{cases}
      \end{split}
    \end{align}
  \end{enumerate} 
\end{proposition}

\begin{proof}
  We start with the characterization of $u_0(x,y)$ for $(x,y)\in
  D\times \Sigma_Y$. Therefore, we consider localized test functions
  $\varphi_\eps(x):=\theta(x)\Psi\left(\frac{x}{\eps}\right)$ with
  $\theta\in C_c^\infty(D)$ and $\Psi\in C^\infty(Y)$ with
  $\text{supp}(\Psi)\subset\Sigma_Y$.  On the one hand we find,
  exploiting the two-scale convergence of $\tilde u^\eps$,
\begin{align*}
  \int_{\Omega}\tilde{u^\eps}(x)\varphi_\eps(x)\,dx\rightarrow \int_D
  \theta(x)\int_{\Sigma_Y} u_0(x,y)\Psi(y)\,dy\,dx\,.
\end{align*}
On the other hand one obtains, using the fact that
$\text{supp}(\varphi_\eps)\subset (\Sigma_\eps\cup \bar K_\eps)$, the
definition of $\tilde{u^\eps}$, the $L^2(\Omega)$-bound for
$\tilde{u^\eps}$ and the vanishing volume fraction of the channels
$K_\eps$,
\begin{align*}
  \int_{\Omega}\tilde{u^\eps}(x)\varphi_\eps(x)\,dx
  =\int_{K_\eps}\tilde{u^\eps}(x)\varphi_\eps(x)\,dx\rightarrow
  0\,.
\end{align*}
Since the test functions were arbitrary, we conclude $u_0(x,y)=0$ for
$(x,y)\in D\times\Sigma_Y$.

We next show that $u_0(x,y)$ is independent of $y$ for $(x,y)\in
D\times Q_Y$.  As before we consider localized test functions
$\varphi_\eps(x):=\theta(x)\Psi\left(\frac{x}{\eps}\right)$ with
$\theta\in C_c^\infty(D)$ and $\Psi\in C^\infty_{per}(Y,\R^n)$ with
$\text{supp}(\Psi)\subset Q_Y$. Multiplying $\nabla u^\eps$ by
$\eps\varphi_\eps$ and integrating by parts gives
\begin{align}
\label{eq:test0}
\eps\int_{\Omega_\eps}\nabla u^\eps(x)\cdot\varphi_\eps(x)\,dx=
-\int_{\Omega_\eps}u^\eps(x)\left[(\nabla_y\cdot\Psi)\left(\frac{x}{\eps}\right)\theta(x) + 
\eps \Psi\left(\frac{x}{\eps}\right)\cdot\nabla_x\theta(x)\right]\,dx\,.
\end{align}
Note that no boundary terms appear due to the compact support of
$\theta$ and $\Psi$. We can now pass to the limit $\eps\rightarrow 0$.
Due to the uniform $L^2$-bound of $\nabla u^\eps$,
cf. \eqref{eq:uniformboundimplied}, the left hand side of
\eqref{eq:test0} vanishes in the limit as $\eps\rightarrow 0$.  On the
right hand side we exploit the two-scale convergence of
$\tilde{u^\eps}$ to find
\begin{align*}
  &\int_{\Omega_\eps}u^\eps(x)\left[(\nabla_y\cdot\Psi)\left(\frac{x}{\eps}\right)\theta(x)
    +
    \eps \Psi\left(\frac{x}{\eps}\right)\cdot\nabla_x\theta(x)\right]\,dx\\
  &\qquad
  =\int_{\Omega}\tilde{u^\eps}(x)\left[(\nabla_y\cdot\Psi)\left(\frac{x}{\eps}\right)\theta(x)
    +
    \eps \Psi\left(\frac{x}{\eps}\right)\cdot\nabla_x\theta(x)\right]\,dx\\
  &\qquad \stackrel{\eps\rightarrow
    0}{\rightarrow}\int_D\theta(x)\int_{Q_Y}u_0(x,y)(\nabla_y\cdot\Psi)(y)\,dy\,dx\,.
\end{align*}
Since the test functions $\theta,\Psi$ were arbitrary, this implies
that $\nabla_y u_0(x,.) = 0$ in the sense of distributions in $Q_Y$
for almost every $x\in D$. Therefore $u_0$ does not depend on $y$ for
$(x,y)\in D\times Q_Y$.  We may therefore write $u_0(x,y)=v(x)$ for
some $v\in L^2(D)$.

For the domains $(\Omega\setminus D)\times Y$ and $D\times R_Y$ 
one proceeds analogously to show that $u_0(x,.)$ does
not depend on $y$.
\end{proof}

\subsection{Two-scale convergence of the gradients}

Due to the uniform bound \eqref{eq:uniformboundimplied} we can find
also a two-scale limit of the sequence $\nabla u^\eps$ (upon extending
by zero).  The subsequent Lemma provides a first characterization of
the two-scale limit. We use the space $H_{per}^1(Q_Y)$ of those
functions in $H^1(Q_Y)$ for which also their periodic extension to
$\R^n$ is locally of class $H^1$.

\begin{lemma}
  \label{lem:twoscalegradients}
  Let $u^\eps\in H^1(\Omega_\eps)$ be a sequence of solutions to
  \eqref{eq:Helmholtzeps} that satisfies the uniform bound
  \eqref{eq:uniformboundassume}. Let $\xi^\eps\in L^2(\Omega;\R^n)$ be
  the trivial extension of $\nabla u^\eps$ by zero and let $v$ be the
  exterior field of Proposition \ref{prop:charactu0}. We assume that
  $\xi^\eps$ converges in two scales to some limit function $\xi_0\in
  L^2(\Omega\times Y)^n$. Then the following holds:
  \begin{enumerate}
	\item The exterior field $v$ is of class $H^1(\Omega)$.
  \item For $x\in\Omega\setminus D$ one has $\xi_0(x,y)=\nabla_x v(x)$.
  \item There exists $v_1\in L^2(D;H_{per}^1(Q_Y))$ such that for $x\in D$ one has
    \begin{align}
      \begin{split}
        \label{eq:charactxi}
	\xi_0(x,y)=
        \begin{cases}
          \nabla_x v(x) + \nabla_y v_1(x,y) & \;\;\text{ for } y\in Q_Y, \\
          0& \;\; \text{ for } y\in R_Y\cup \Sigma_Y.
        \end{cases}
      \end{split}
    \end{align}
  \end{enumerate}
\end{lemma}

We note that the lemma implies for the interior of the resonators
\begin{align}
  \label{eq:charactxi2}
  \nabla u^\eps\,{\bf 1}_{R_\eps}\stackrel{2s}{\weakto}0\quad\text{ as
  }\eps\rightarrow 0\,.
\end{align}

\begin{proof} {\em Step 1. Regularity of $v$.} We consider the domain
  $\hat\Omega_\eps:=\Omega\setminus(\bar\Sigma_\eps\cup \bar
  K_\eps\cup \bar R_\eps)$, which is obtained from $\Omega$ by
  removing the union of obstacles $\Sigma_\eps$, slits $K_\eps$ and
  interior regions $R_\eps$.  In particular, the perforation of
  $\hat\Omega_\eps$ in each periodicity cell is a Lipschitz domain
  without substructure.  We construct a sequence $\hat v^\eps\in
  H^1(\Omega)$ by setting $\hat v^\eps:=u^\eps$ in $\hat\Omega_\eps$
  and extending to $\Omega$. Actually, it is well known that there
  exists a family of extension operators $\mathcal{P}_\eps:
  H^1(\hat\Omega_\eps)\rightarrow H^1(\Omega)$ such that
  \begin{align*}
    \left\|\mathcal{P}_\eps u^\eps\right\|_{H^1(\Omega)}\leq C\|u^\eps\|_{H^1(\hat\Omega_\eps)}
  \end{align*}
  for some $C>0$ independent of $\eps$, see \cite{Cioranescu1999},
  Chapter 1. Essentially, $\mathcal{P}_\eps$ is defined by using in
  each perforation the harmonic extension of the boundary
  values. Hence, $\hat v^\eps:=\mathcal{P}_\eps u^\eps$ is uniformly
  bounded in $H^1(\Omega)$.  We therefore find, up to a subsequence, a
  limit function $\hat v\in H^1(\Omega)$ such that $\hat v^\eps$
  converges in two scales and strongly in $L^2(\Omega)$ to the
  ($y$-independent) function $\hat v=\hat v(x)$.

  Since $\hat v^\eps=u^\eps$ in $\hat\Omega_\eps$, the strong limit
  $\hat v$ coincides with the two-scale limit of $u^\eps$ in the
  exterior of the resonators, i.e. $\hat v=v$. This proves $v\in
  H^1(\Omega)$.

  \medskip {\em Step 2. Characterization outside $D$.}  Outside of
  $D$, the elliptic equation for $u^\eps$ implies that the solution
  sequence is locally $H^2$-bounded. Therefore, the distributional
  convergence $\nabla u^\eps\to \nabla v$ is locally a strong
  $L^2$-convergence.  In this case the two-scale limit of gradients
  coincides with the strong limit.
	
  \medskip {\em Step 3. Characterization in $D$: The case $y\in
    Q_Y\cup\Sigma_Y$.}  Concerning the characterization for $(x,y)\in
  D\times Q_Y$ we refer to a standard argument, which can be found
  e.g.\,in the proof of Theorem 2.9 in \cite{Allaire1992}. It provides
  the existence of a two-scale function $v_1\in L^2(D;
  H_{per}^1(Q_Y))$ such that
  \begin{align}
    \label{eq:gradientcharact1}
    \xi_0(x,y)=\nabla_x v(x) + \nabla_y v_1(x,y)\quad\text{for }
    (x,y)\in D\times Q_Y.
  \end{align}  
  We emphasize that characterization \eqref{eq:gradientcharact1}
  heavily relies on the fact that the union of exterior domains is
  connected.

  The claim for $y\in \Sigma_Y$ follows as in the proof of Proposition
  \ref{prop:charactu0}: We exploit the vanishing volume fraction of
  the channels and the uniform bound for $\nabla u^\eps$.
	
  \medskip {\em Step 4. Characterization in $D$: The case $y\in R_Y$.}
  For $y\in R_Y$ we argue as follows: Let $\theta\in C_c^\infty(D)$
  and $\Psi\in C_c^\infty(R_Y,\R^n)$ be arbitrary. In three space
  dimensions, $n=3$, we exploit the relation $\curl(\nabla u^\eps)=0$,
  while in two space dimensions the $\curl$-operator has to be
  replaced by a rotated divergence.  In the following, we perform the
  calculations only for $n=3$.  Multiplying the identity $\curl(\nabla
  u^\eps)=0$ with $\eps\,\theta(x)\Psi\left(\frac{x}{\eps}\right)$ and
  integrating by parts one finds
  \begin{align*}
    0&=\int_{\Omega_\eps}\curl(\nabla u^\eps)(x)\cdot\eps\Psi\left(\frac{x}{\eps}\right)\theta(x)\,dx\\
    &=\int_{\Omega_\eps}\nabla u^\eps(x)\cdot
    \left[\curl_y\Psi\left(\frac{x}{\eps}\right)\theta(x)-\eps\Psi\left(\frac{x}{\eps}\right)\wedge\nabla_x\theta(x)\right]\,dx\\
    &\stackrel{\eps\rightarrow 0}{\rightarrow}
    \int_D\int_{R_Y}\xi_0(x,y)\cdot\curl_y\Psi(y)\theta(x)\,dy\,dx\,.
  \end{align*}
  Since $\theta\in C_c^\infty(D)$ was arbitrary we conclude that for
  a.e.\,$x\in D$ and every $\Psi\in C_c^\infty(R_Y,\R^3)$ there holds
  \begin{align*}
    \int_{R_Y} \xi_0(x,y)\cdot\curl_y\Psi(y)\,dy=0\,.
  \end{align*}
  We obtained that the (distributional) curl of $\xi_0$ vanishes.
  Since the interior domain $R_Y$ is simply connected,
  $\xi_0(x,\cdot)$ must be a gradient: There exists a potential
  $w_1\in L^2(D; H^1(R_Y))$ such that $\xi_0(x,\cdot)=\nabla_y
  w_1(x,\cdot)$ for a.e. $x\in D$. With an analogous calculation the
  same result can be obtained also in two space dimensions, $n=2$.

  We next show that $w_1(x,\cdot)$ is constant for a.e.\,$x\in
  D$. This implies $\xi_0(x,y) = \nabla_y w_1(x,y)=0$; with that, the
  lemma is proven.

  We consider microscopic test-functions $\psi\in C^\infty(\bar R_Y)$
  with vanishing values at the entrance of the channel.  More
  precisely, denoting by $\Gamma_Y^{\eps,R}:=\bar R_Y\cap \bar
  K_Y^\eps=\{y_R\}\times \overline{B^{n-1}_{\alpha\eps^p}(0)}$ the interface
  between $R_Y$ and the channel $K_Y^\eps$ we define, for $\delta>0$
  fixed, the set
\begin{align}
\mathcal{A}_\delta(R_Y):=\left\{\psi\in C^\infty(\bar R_Y)\,|\, \psi=0\,\text{ on } \{y_R\}\times B_\delta^{n-1}(0) \right\}.
\end{align}
Note that for $\delta>0$ fixed and $\eps$ sufficiently small each
$\psi_\delta\in \mathcal{A}_\delta(R_Y)$ satisfies $\psi_\delta=0$ on
$\Gamma_Y^{\eps,R}$.  Let now $\theta\in C_c^\infty(D)$ and
$\psi_\delta\in \mathcal{A}_\delta(R_Y)$ be arbitrary. We multiply
$-\Delta u^\eps$ with
$\eps\,\theta(x)\psi_\delta\left(\frac{x}{\eps}\right)$ and integrate
by parts to find
\begin{align}
  \begin{split}
    \label{eq:testnabla}
    &\eps\int_{\Omega_\eps}(-\Delta u^\eps)(x)\, \theta(x)\psi_\delta\left(\frac{x}{\eps}\right)\,dx\\
    &\qquad =\int_{R_\eps}\nabla u^\eps\cdot\left[\eps
      \nabla_x\theta(x)\psi_\delta\left(\frac{x}{\eps}\right) +
      \theta(x)\nabla_y\psi_\delta\left(\frac{x}{\eps}\right)\right]\,dx\,.
  \end{split}
\end{align}
Note that no boundary terms appear due to the Neumann boundary
conditions of $u^\eps$ and the fact that the test function
$\Psi_\delta$ vanishes at the interface between $R_Y$ and the
channel. We next pass to the limit in \eqref{eq:testnabla}. Exploiting
that $\Delta u^\eps$ is uniformly bounded in $L^2(\Omega_\eps)$ we
find that the left hand side of \eqref{eq:testnabla} vanishes in the
limit as $\eps\rightarrow 0$.  On the right hand side we use the
uniform boundedness of $\nabla u^\eps$ to conclude that the first
term vanishes in the limit, while for the second term we obtain
\begin{align*}
  \int_{\Omega_\eps}\nabla u^\eps\cdot\theta(x)\nabla_y\psi_\delta\left(\frac{x}{\eps}\right)\,dx
  &\stackrel{\eps\rightarrow 0}{\rightarrow}
  \int_D\int_{R_Y}\xi_0(x,y)\cdot\nabla_y\psi_\delta(y)\theta(x)\,dy\,dx\\
  &=\int_D\int_{R_Y}\nabla_y
  w_1(x,y)\cdot\nabla_y\psi_\delta(y)\theta(x)\,dy\,dx\,.
\end{align*}
In the last line we used the characterization $\xi_0(x,y)=\nabla_y
w_1(x,y)$. Since $\theta$ was arbitrary we conclude, for a.e. $x\in D$
and every test function $\psi_\delta(y)\in \mathcal{A}_\delta(R_Y)$,
\begin{align}
\label{eq:harmonicdelta}
\int_{R_Y}\nabla_y w_1(x,y)\cdot\nabla_y\psi_\delta(y)\,dy=0\,.
\end{align}
In the above equality the test functions $\psi_\delta$ are restricted
to the set $\mathcal{A}_\delta(R_Y)$.  We claim that
\eqref{eq:harmonicdelta} holds for all functions $\psi\in
C^\infty(\bar R_Y)$.  Indeed, in the limit $\delta\rightarrow 0$ the
set $\{y_R\}\times B^{n-1}_\delta(0)$ shrinks to a set of vanishing
$H^1$-capacity. In particular, for every $\psi\in C^\infty(\bar R_Y)$
there exists an approximating sequence $\psi_\delta \in
\mathcal{A}_\delta(R_Y)$ with $\psi_\delta\rightarrow \psi$ in
$H^1(R_Y)$ as $\delta\rightarrow 0$. We therefore conclude that for every $\psi\in
C^\infty(\bar R_Y)$
\begin{align*}
\int_{R_Y}\nabla_y w_1(x,y)\cdot\nabla_y\psi(y)\,dy=0\,.
\end{align*}
We obtain that $w_1(x,\cdot)$ is a solution to
\begin{align*}
-\Delta_y w_1(x,\cdot)&=0\quad\text{ in } R_Y,\\
\partial_n w_1(x,\cdot)&=0 \quad\text{ on } \partial R_Y\,.
\end{align*}
All solutions to this elliptic problem are constant functions, as can
be shown by testing with $w_1$. We obtain $\xi_0(x,y)=\nabla_y
w_1(x,y)=0$ for $(x,y)\in D\times R_Y$, which concludes the proof.
\end{proof}

In the next lemma we relate, via cell problems, the two-scale function
$v_1$ (the exterior two-scale corrector) to the field $v$ (which
represents average values outside the resonators). The procedure
follows the standard arguments, our interest is to show rigorously
that the channels do not affect the equations for the exterior
corrector function $v_1$. As before, $n$ denotes exterior normal
vectors on boundaries.

\begin{lemma}[Characterization of $v_1$]
  Let $v$ and $v_1$ be as in Proposition \ref{prop:charactu0} and
  Lemma \ref{lem:twoscalegradients}. Then the microscopic function
  $v_1$ can be written as
  \begin{align}
    \label{eq:representv1}
    v_1(x,y)=\sum_{i=1}^n \chi_i(y)\partial_{x_i}v(x)\,,
  \end{align}
  where the shape functions $\chi_i\in H_{per}^1(Q_Y)$ with
  $i\in\{1,...,n\}$ satisfy the following cell problem with Neumann
  boundary conditions on $\partial Q_Y\setminus\partial Y$ and
  periodicity boundary conditions on $\del Y$:
\begin{align}
\begin{split}
  \label{eq:cellproblemw}
  \Delta_y \chi_i & = 0\quad\text{in }Q_Y\,,\\
  \partial_n \chi_i & =- e_i\cdot n\quad\text{on }\partial Q_Y\setminus\partial Y\,.
\end{split}
\end{align}
\end{lemma}
\begin{proof}
  For an arbitrary microscopic test function $\psi\in H_{per}^1(Q_Y)$
  we will prove that there holds, for a.e.\,$x\in D$,
  \begin{align}
    \label{eq:toshow1}
    \int_{Q_Y}\left[\nabla_x v(x)+\nabla_y v_1(x,y)\right]\cdot\nabla_y\psi(y)\,dy=0\,.
  \end{align}
  Equation \eqref{eq:toshow1} implies the claim of the lemma. Indeed,
  \eqref{eq:toshow1} is the weak formulation of the Neumann problem
  \begin{align*}
    -\Delta_y v_1(x,y) =
    -\nabla_y\cdot\left[\nabla_x v(x)+\nabla_y v_1(x,y)\right]
    &=0\quad\text{in } Q_Y\,,\\
    (\nabla_x v(x) + \nabla_y v_1(x,y))\cdot n&=0\quad\text{on
    } \partial Q_Y\setminus\partial Y\,,
  \end{align*}
  supplemented with periodicity boundary conditions on $\partial Y$. By linearity of
  this problem, $v_1$ depends linearly on $\nabla_x v$. This yields
  the representation formula \eqref{eq:representv1}.

  \smallskip {\it Derivation of \eqref{eq:toshow1}:} In the following
  we fix a small parameter $\delta>0$ in such a way that the
  $\delta$-neighbourhood of $Q_Y$ in $Y$ does not intersect the
  interior domain $R_Y$, i.e. $B_{\delta}(Q_Y)\cap R_Y=\emptyset$.
  Let $\theta\in C_c^\infty(D)$ be arbitrary. We consider microscopic
  test-functions $\psi\in H^1_{per}(Y)$ with $\text{supp}(\psi)\subset
  B_{\delta}(Q_Y)$.  We emphasize that all $H^1_{per}(Q_Y)$-functions
  are admissible, we do not prescribe any boundary conditions on the
  channel opening. We multiply the Helmholtz equation
  \eqref{eq:Helmholtzeps} with
  $\varphi_\eps(x):=\eps\theta(x)\psi\left(\frac{x}{\eps}\right)$ to
  find
  \begin{align}
    \label{eq:testphi2}
    \int_{\Omega_\eps}\nabla u^\eps(x)\cdot\left[\nabla_y
      \psi\left(\frac{x}{\eps}\right)\theta(x) +
      \eps\nabla_x\theta(x)\psi\left(\frac{x}{\eps}\right)
    \right]\,dx=\eps\omega^2\int_{\Omega_\eps}u^\eps(x)\theta(x)\psi\left(\frac{x}{\eps}\right)\,dx\,.
  \end{align} 
  Due to the uniform $H^1$-bound on $u^\eps$ the second term on the
  left hand side of \eqref{eq:testphi2} and the right hand side vanish
  in the limit $\eps\rightarrow 0$.  Concerning the first term on the
  left hand side we calculate, using that $\text{supp}(\psi)\subset
  Q_Y\cup \Sigma_Y$ and thus
  $\Omega_\eps\cap\,\text{supp}\left(\psi\left(
      \frac{\cdot}{\eps}\right)\theta\right)\subset (Q_\eps\cup
  K_\eps)$,
  \begin{align*}
    &\int_{\Omega_\eps}\nabla u^\eps(x)\cdot\nabla_y \psi\left(\frac{x}{\eps}\right)\theta(x)\,dx\\
    &\qquad =\int_{Q_\eps}\nabla u^\eps(x)\cdot\nabla_y \psi\left(\frac{x}{\eps}\right)\theta(x)\,dx
    +\int_{K_\eps}\nabla u^\eps(x)\cdot\nabla_y \psi\left(\frac{x}{\eps}\right)\theta(x)\,dx\\
    &\qquad \stackrel{\eps\rightarrow 0}{\rightarrow}\int_D\int_{Q_Y}(\nabla_x v(x) + \nabla_y v_1(x,y))
    \cdot\nabla_y \psi(y)\theta(x)\,dy\,dx\,.
  \end{align*}
  In the last line we exploited the uniform $H^1$-bound of $u^\eps$
  and the vanishing volume fraction of the channels $K_\eps$.
  Moreover, we used characterization \eqref{eq:charactxi} for the
  two-scale limit of $\nabla u^\eps$. We obtain
  \begin{align*}
    \int_D\int_{Q_Y}(\nabla_x v(x) + \nabla_y v_1(x,y))
    \cdot\nabla_y \psi(y)\theta(x)\,dy\,dx=0
  \end{align*} 
  and hence, since $\theta\in C_c^\infty(D)$ was arbitrary, the claim
  \eqref{eq:toshow1}.
\end{proof}


\section{Proof of the main result}
In this section we prove that the external field $v$ of Proposition
\ref{prop:charactu0} satisfies the effective Helmholtz equation
\eqref{eq:effectiveHelmholtz-1}. Let us start by formulating the
strong version of our main result.

\begin{theorem}[Characterization of two-scale limits of solutions]
  \label{thm:main-2} 
  Let $u^\eps\in H^1(\Omega_\eps)$ be as in Theorem \ref {thm:main-1}:
  For every $\eps$, the function $u^\eps$ solves the Helmholtz
  equation \eqref{eq:Helmholtzeps} in a domain $\Omega_\eps$, which is
  obtained by removing from $\Omega\subset\R^n$ a family of small
  resonators. We assume that the family $u^\eps$ satisfies the uniform
  bound \eqref{eq:uniformboundassume}.  Due to this boundedness, we
  can extract a two-scale convergent subsequence and study $v\in
  L^2(\Omega)$ of Proposition \ref{prop:charactu0}.

  Then the field $v$ is of class $v\in H^1(\Omega)$ and it is a
  solution to the effective Helmholtz equation
  \begin{align}
    \label{eq:effectiveHelmholtznew1}
    -\nabla\cdot\left(A_{*}\nabla v\right) 
    &=\omega^2\Lambda\, v\quad\text{ in }\Omega\,,\\
    \label{eq:effectiveHelmholtznew2}
    v&=g\quad\quad\quad\text{ on }\partial\Omega\,.
  \end{align}
  The effective coefficients are $\Lambda(x)=1$ and $A_{*}(x) =
  \text{{\bf 1}}_{\R^n}$ for $x\in\Omega\setminus D$, whereas for
  $x\in D$ the factor $\Lambda(x) = \Lambda_{\eff}$ and the matrix
  $A_{*}(x) = A_{\eff}$ are defined in \eqref{eq:Lambda-1} and
  \eqref{eq:effectiveA}.
\end{theorem}
In order to prove Theorem \ref{thm:main-2} we have to introduce an
additional quantity, namely the current $j_*$. To define this
additional effective quantity, we first observe that a rescaled flux
in the channels is bounded in $L^1(\Omega)$: The sequence
\begin{align}
  \label{eq:currenteps}
  j^\eps(x):=-\frac{1}{L\eps}{\bf 1}_{K_\eps}(x)\partial_{x_1}u^\eps(x)\,,
\end{align}
where ${\bf 1}_{K_\eps}$ is the characteristic function of the
channels, satisfies
\begin{align*}
  \int_{\Omega}|j^\eps|=\int_{K_\eps}\frac{1}{L\eps}|\partial_{x_1}u^\eps|
  \leq
  \frac{|K_\eps|^{\frac12}}{L\eps}\left(\int_{K_\eps}|\partial_{x_1}u^\eps|^2\right)^{\frac12}\leq
  C,
\end{align*}
where $C$ is independent of $\eps$. In the last inequality we
exploited that the volume of the channels is of order $\eps^2$
(opening area times length times number is $\eps^3\eps \eps^{-2}$ for
$n=2$ and $(\eps^2)^2\eps \eps^{-3}$ for $n=3$), and that $u^\eps$ is
uniformly bounded in $H^1(\Omega_\eps)$.

In view of this estimate we find a subsequence $\eps\rightarrow 0$ and
a limit $j_*\in \calM(\bar\Omega)$ such that $j^{\eps}\rightarrow j_*$
weakly star in the space of Radon measures along the subsequence.  
\begin{definition}[The current $j_*$]
  \label{def:current}
  We define the current $j_*\in\mathcal{M}(\bar\Omega)$ as the weak
  star limit of $j^{\eps}$ in the space of Radon measures,
  i.e. through
  \begin{align}
    \label{eq:currentdef}
    -\frac{1}{L{\eps}}\int_{K_{\eps}}\partial_{x_1}u^{\eps}(x)
    \theta(x)\,dx\quad \stackrel{{\eps}\rightarrow
      0}{\rightarrow}\quad \int_{\bar\Omega} \theta(x)\,dj_*(x)
  \end{align}
  for all test-functions $\theta\in C(\bar\Omega)$.
\end{definition}
At a later stage, cf. Proposition \ref{prop:1}, we will see that the
measure $j_*\in\mathcal{M}(\Omega)$ is absolutely continuous with
respect to the Lebesgue measure and that its density coincides, up to
a prefactor, with the interior field $w(\cdot)$ from characterization
\eqref{eq:charactvnew}. In particular, the limit $j_*$ does not depend
on the choice of the subsequence $\eps\rightarrow 0$.

Note that, by definition of the set $K_\eps$, the current $j_*$
vanishes outside the perforated region $\bar D$. The integral on the
right hand side of \eqref{eq:currentdef} may be replaced by an
integral over $\bar D$:
\begin{align}
  \label{eq:restrictionj}
  \int_{\bar\Omega} \theta(x)\, dj_*(x)=\int_{\bar{D}}
  \theta(x)dj_*(x)\quad\text{ for all }\theta\in C(\bar\Omega)\,.
\end{align}

The proof of Theorem \ref{thm:main-2} consists of three steps: (1)
Establish a relation between the current $j_*$ and the interior field
$w$. This is achieved in Proposition \ref{prop:1}. (2) Derivation of a
geometric flow rule, which relates the current $j_*$ to the difference
$v-w$ between exterior field $v$ and interior field $w$. This is the
result of Proposition \ref{prop:geometricflow}. (3) Derivation of a
Helmholtz equation for $v$, cf.\,Proposition \ref{prop:effectiveeq}.

\subsection{Relation between current and the interior field}

In this subsection we prove that the current $j_*$ can be expressed in
terms of the interior field $w$. We recall that $V=|R_Y|$ is the
relative volume of the single resonator domain.

\begin{proposition}[Relation between the current and the interior field]
  \label{prop:1}
  In the situation of Theorem \ref{thm:main-2}, let the current
  $j_*\in\mathcal{M}({\bar\Omega})$ be as in Definition
  \ref{def:current} and let $w\in L^2(D)$ be the interior field of
  Proposition \ref{prop:charactu0}.  Then for every function
  $\theta\in C_c^\infty(\Omega)$ there holds
  \begin{align}
    \label{eq:weakformj}
    \int_\Omega\theta(x)\,dj_*(x)=V\omega^2\int_D \theta(x)w(x)\,dx\,.
  \end{align} 
  In particular, taking into account \eqref{eq:restrictionj}, one obtains that $j_*=j\,dx$, where $j$ is the $L^2(\Omega)$-function
  \begin{align}
    \label{eq:relationjw}
    j(x)=
    \begin{cases}
      V\omega^2 w(x)&\text{ for } x\in D\,, \\
      0&\text{ for } x\in \Omega\setminus D\,.
    \end{cases}
  \end{align}
\end{proposition} 

\begin{proof}
  Let $\theta\in C_c^\infty(\Omega)$ be arbitrary. Our aim is to
  multiply the Helmholtz equation \eqref{eq:Helmholtzeps} with an
  oscillatory test function of the form
  $\varphi_\eps(x):=\theta(x)\Phi_\eps\left(\frac{x}{\eps}\right){\bf
    1}_{Y^\eps}(x)$, where $Y^\eps:=\bigcup_{k\in\mathcal{K}}
  Y_k^\eps$ denotes the union of $\eps$-cubes that are contained in
  the resonator region $D$.  The function $\Phi_\eps: Y\setminus
  \bar\Sigma_Y^\eps\rightarrow \R$ is defined as follows:
\begin{align}
\begin{split}
\label{eq:testfctaffine}
\Phi_\eps(y)=
\begin{cases}
    L \;\; &\text{ for } y\in R_Y\text{ (interior)},\\
    0\;\, &\text{ for } y\in Q_Y\text{ (exterior)},\\
    \Phi(y_1)&\text{ for }  y\in K_Y^\eps \text{ (channel)},
\end{cases}
\end{split}
\end{align}
where $L$ is the length of the channel and $\Phi(\cdot)$ is affine on
the interval $(y_R,y_Q)$ with $\Phi(y_R) = L$ and $\Phi(y_Q) = 0$. In
particular, $\Phi_\eps\in H^1(Y\setminus \bar\Sigma^\eps_Y)$ with
\begin{align*}
  \nabla_y\Phi_\eps(y)=- e_1\,{\bf 1}_{K_Y^\eps}(y)\,,
\end{align*}
where $e_1$ denotes the first unit vector. We note that the
oscillatory test function $\varphi_\eps$ is of class
$H^1(\Omega_\eps)$, since the jump set $\partial Y^\eps$ of the
function ${\bf 1}_{Y^\eps}$ is contained in the set where
$\Phi_\eps\left(\frac{\cdot}{\eps}\right)$ vanishes.  We now multiply
the Helmholtz equation \eqref{eq:Helmholtzeps} with $\varphi_\eps$ and
integrate by parts to find
\begin{align}
  \label{eq:test1}
  \int_{\Omega_\eps}\nabla u_\eps(x)\cdot\nabla
  \varphi_\eps(x)\,dx=\omega^2\int_{\Omega_\eps}
  u^\eps(x)\varphi_\eps(x)\,dx\,.
\end{align} 
Concerning the left hand side of \eqref{eq:test1} we calculate
\begin{align*}
  &\int_{\Omega_\eps}\nabla u_\eps(x)\cdot\nabla \varphi_\eps(x)\,dx\\
  &\qquad =
  \int_{\Omega_\eps}\nabla u_\eps(x)\cdot\nabla\theta(x)\Phi_\eps\left(\frac{x}{\eps}\right){\bf 1}_{Y^\eps}(x)\,dx -\frac{1}{\eps} \int_{K_\eps}\partial_{x_1}u^\eps(x)\theta(x)\,dx\\
  &\qquad =L\int_{R_\eps}\nabla u_\eps(x)\cdot\nabla\theta(x)\,dx +
  \int_{K_\eps}\nabla
  u_\eps(x)\cdot\nabla\theta(x)\Phi_\eps\left(\frac{x}{\eps}\right)\,dx\\
  &\qquad\qquad -\frac{1}{\eps}
  \int_{K_\eps}\partial_{x_1}u^\eps(x)\theta(x)\,dx\,.
\end{align*}
In the limit $\eps\rightarrow 0$, the first term vanishes due to
\eqref{eq:charactxi2}. The second term vanishes due to the uniform
bound \eqref{eq:uniformboundimplied} and the vanishing volume fraction
of the channels.  In the third term we use the definition of $j_*$ in 
\eqref{eq:currentdef}. Together, we find
\begin{align*}
  \int_{\Omega_\eps}\nabla u_\eps(x)\cdot\nabla
  \varphi_\eps(x)\,dx
  \quad \stackrel{\eps\rightarrow 0}{\rightarrow}\quad
  L\int_{\bar\Omega} \theta(x)\,dj_*(x)\,.
\end{align*}
Concerning the right hand side of \eqref{eq:test1} we calculate, using
the properties of the microscopic test-function $\Phi_\eps$,
\begin{align*}
  &\omega^2\int_{\Omega_\eps} u^\eps(x)\varphi_\eps(x)\,dx
  =L\omega^2\int_{R_\eps}u^\eps(x)\theta(x)\,dx
  + \omega^2\int_{K_\eps}u^\eps(x)\varphi_\eps(x)\,dx\\
  &\qquad \stackrel{\eps\rightarrow 0}{\rightarrow}
  L\omega^2\int_{D}\int_{R_Y}w(x)\theta(x)\,dy\,dx
  =LV\omega^2\int_{D}w(x)\theta(x)\,dx\,.
\end{align*}
In the second line we again exploited the vanishing volume fraction of
the channels and characterization \eqref{eq:charactvnew} of the
two-scale limit $u_0$. 

Our calculations show that \eqref{eq:test1} provides, in the limit
$\eps\to 0$ and upon dividing by $L$, relation \eqref
{eq:weakformj}. This concludes the proof of the proposition.
\end{proof}

\subsection{Geometric flow rule: A second relation for the current}

In this section we establish the geometric flow rule, which relates
the current $j_*$ to the difference between the exterior field $v$ and
the interior field $w$.

\begin{proposition}[Geometric flow rule]
  \label{prop:geometricflow}
  In the situation of Theorem \ref{thm:main-2}, let $j_*=j\,dx$ be the
  current from Definition \ref{def:current}, let $w$ be the interior
  field and $v$ the exterior field of Proposition
  \ref{prop:charactu0}.  Then there holds, for $x\in D$,
  \begin{align}
    \label{eq:geometricflow}
    j(x)=-\frac{A}{L}(v(x)-w(x))\,.
  \end{align}
\end{proposition}

We call \eqref{eq:geometricflow} a geometric flow rule, since it
establishes a relation between a suitable average of gradients
$\del_{x_1} u^\eps$ on the left hand side with an averaged slope
$(v-w)/L$, that is to be expected by the values near the end-points of
the channel.

The geometric flow rule together with Proposition \ref{prop:1}
provides already a relation between the exterior field $v$ and the
interior field $w$.

\begin{remark}
  Combining the flux relation \eqref{eq:relationjw} with the geometric
  flow rule \eqref{eq:geometricflow} we obtain that, for almost every
  $x\in D$:
  \begin{align}
    \label{eq:relationvw}
    -\frac{A}{VL}v(x)&=\left(\omega^2-\frac{A}{VL}\right)w(x)\,.
  \end{align}
  In particular, resonances of the system occur for frequencies
  $\omega^2$ that are close to the ratio $A/(VL)$.
\end{remark}

\begin{proof}[Proof of the geometric flow rule, Proposition
  \ref{prop:geometricflow}]
  We will prove that for arbitrary Lipschitz domains $E\subset D$ the
  limit measure $j_* = j\, dx$ satisfies
  \begin{equation}
    \label{eq:Aim-geometric}
    \int_E\, j(x)\,dx=-\frac{A}{L}\int_E(v(x)-w(x))\,dx\,.
  \end{equation}
  Once this is shown, we have verified \eqref {eq:geometricflow} and
  hence the proposition. We will use essentially the same
  test-function as in the proof of Proposition \ref{prop:1}. But while in Proposition
  \ref{prop:1} we concluded results from the Helmholtz equation, we use the test-functions here in a more elementary
  way: We want to compare values of functions with averages of
  derivatives.
	
  \smallskip {\em Step 1. The two-scale test function.} We use the
  microscopic test-function $\Phi_\eps:Y\setminus \bar\Sigma_Y^\eps \to
  \R$ that was defined in \eqref{eq:testfctaffine}. In particular,
  $\Phi_\eps\in H^1(Y\setminus \bar\Sigma^\eps_Y)$ with
  \begin{align*}
    \nabla_y\Phi_\eps(y) = - e_1\,{\bf 1}_{K_Y^\eps}(y)\,.
  \end{align*}
  To construct a macroscopic test-function, we define $I_\eps := \{
  k\in \Z^n | Y_k^\eps \subset E\}$, the set of all indices $k$ such
  that the cell $Y_k^\eps$ is contained in the test-set $E\subset
  D$. The number of elements of $I_\eps$ is of order $M_\eps :=
  |I_\eps| = O(\eps^{-n})$. As a slightly smaller test-set we use
  $Y_E^{\eps} := \bigcup_{k\in I_\eps} Y_k^\eps \subset E$, the union
  of $\eps$-cells that are contained in $E$.  With the characteristic
  function $\Theta^\eps_E:\Omega\to [0,1]$, $\Theta^\eps_E(x) = 1$ for
  $x\in Y_E^{\eps}$ and $\Theta^\eps_E(x)=0$ for $x\not\in
  Y_E^{\eps}$, we set
  \begin{equation}
    \label{eq:test-fct-fhi}
    \fhi_\eps(x) := \Theta^\eps_E(x) \Phi_\eps\left( \frac{x}{\eps} \right)\,.
  \end{equation}
  We note that $\fhi_\eps$ is of the class $H^1$, since the jump set $\del Y_E^\eps$ of the function $\Theta^\eps_E$ is contained in
  the set where the function $\Phi_\eps\left( \frac{\cdot}{\eps}\right)$
  vanishes.

  \smallskip {\em Step 2.  First calculation of $B_\eps$.} The proof
  of the proposition consists of calculating, in two different ways,
  the following expression $B_\eps$:
  \begin{align*}
    B_\eps &:= \int_{\Omega_\eps} \nabla u^\eps(x)\cdot \nabla \fhi_\eps(x)\,dx\,.
  \end{align*}
  On the one hand, we can calculate with the definition of $j_*=j\,dx$:
  \begin{align*}
    B_\eps &= \frac{1}{\eps}\int_{Y_E^\eps\cap R_\eps} 
		\nabla u^\eps(x)\cdot \nabla_y \Phi_\eps\left(\frac{x}{\eps}\right)\,dx +
    \frac{1}{\eps}\int_{Y_E^\eps\cap K_\eps} \nabla u^\eps(x)\cdot \nabla_y \Phi_\eps\left(\frac{x}{\eps}\right)\,dx\\
    &\qquad + \frac{1}{\eps}\int_{Y_E^\eps\cap Q_\eps} \nabla u^\eps(x)\cdot \nabla_y \Phi_\eps\left(\frac{x}{\eps}\right)\,dx\\
    &= -\frac{1}{\eps} \int_{Y_E^\eps\cap K_\eps} \del_1 u^\eps(x)\,dx
    \to L \int_E j(x)\, dx\,.
  \end{align*}
  In the first equality we decomposed the integral and in the second
  equality we exploited that $\nabla_y\Phi_\eps$ is either $0$ or
  $-e_1$. In the last step we used the definition of the
  limit measure $j_* = j\, dx$ and the fact that the volume of
  $E\setminus Y^\eps_E$ vanishes in the limit $\eps\rightarrow 0$.

  \smallskip {\em Step 3. Second calculation of $B_\eps$.}  In order
  to prepare the second calculation of $B_\eps$ (which is based on an
  integration by parts), we have to define the interface sets: We
  recall that the channel in the periodicity cell is $K_Y^\eps =
  (y_R,y_Q) \times B^{n-1}_{\alpha\eps^p}(0) \subset Y$, the interface
  to the inner set $R_Y$ is therefore $\Gamma_Y^{\eps,R} := \bar
  K_Y^\eps \cap \bar R_Y = \{y_R\}\times
  \overline{B^{n-1}_{\alpha\eps^p}(0)}$. Analogously, the interface to the outer
  set $Q_Y$ is $\Gamma_Y^{\eps,Q} := \bar K_Y^\eps \cap \bar Q_Y =
  \{y_Q\}\times \overline{B^{n-1}_{\alpha\eps^p}(0)}$.

  The function $\Phi_\eps$ has the gradients $0$ and $-e_1$ on the two
  sides of $\Gamma_Y^{\eps,R}$, and it has the gradients $-e_1$ and
  $0$ on the two sides of $\Gamma_Y^{\eps,Q}$.  Therefore, the jumps
  (always right trace minus left trace) are given by
  \begin{equation}
    \label{eq:jump-Phi}
    \left[ \nabla \Phi_\eps \right]_{\Gamma_Y^{\eps,R}} = - e_1\,,\qquad
    \left[ \nabla \Phi_\eps \right]_{\Gamma_Y^{\eps,Q}} =  e_1\,.
  \end{equation}

  We now calculate the number $B_\eps$ by performing an integration by
  parts in all the three integrals, $Y_E^\eps\cap R_\eps$,
  $Y_E^\eps\cap Q_\eps$, and $Y_E^\eps\cap K_\eps$.  Since the
  test-function is harmonic, $\Delta \fhi_\eps=0$, in all the three sets,
  we obtain, denoting by $\Gamma_R^\eps:=\bigcup_{k\in I_\eps}\eps(k+
  \Gamma_Y^{\eps,R})$ and by $\Gamma_Q^\eps:=\bigcup_{k\in
    I_\eps}\eps(k+\Gamma_Y^{\eps,Q})$ the union of interfaces,
  \begin{align*}
    B_\eps &= - \int_{\Gamma^\eps_R} u^\eps\ e_1\cdot
    \left[ \nabla \fhi_\eps \right]_{\Gamma^\eps_R}\, d\calH^{n-1} -
    \int_{\Gamma^\eps_Q} u^\eps\ e_1\cdot
    \left[ \nabla \fhi_\eps \right]_{\Gamma^\eps_Q}\, d\calH^{n-1}\\
    &= \frac1{\eps} \int_{\Gamma^\eps_R} u^\eps\, d\calH^{n-1} -
    \frac1{\eps} \int_{\Gamma^\eps_Q} u^\eps \, d\calH^{n-1}\,.
  \end{align*}
  In the remainder of this proof we want to compare the first integral
  with the values of $u^\eps$ inside the resonators, i.e.\,with $w$,
  and the second integral with the values outside the resonators,
  i.e.\,with $v$. More precisely, we claim that
  \begin{equation}
    \label{eq:claim-geom-lim}
    \frac1{\eps} \int_{\Gamma^\eps_R} u^\eps\, 
    d\calH^{n-1} \to A \int_E w\qquad\text{and}\qquad
    \frac1{\eps} \int_{\Gamma^\eps_Q} u^\eps\, 
    d\calH^{n-1} \to A \int_E v
  \end{equation}
  as $\eps\to 0$. Once \eqref {eq:claim-geom-lim} is shown, the proof
  of the proposition is complete: In combination with Step 2 we have
  found
  \begin{align*}
    L \int_E j(x)\, dx = \lim_{\eps\to 0} B_\eps = A \int_E (w -
    v)(x)\, dx
  \end{align*}
  and hence \eqref {eq:Aim-geometric}.

  \smallskip {\em Step 4. Verification of \eqref {eq:claim-geom-lim}.}
  In order to verify the limits \eqref {eq:claim-geom-lim} we want to
  construct an averaged function, rescale and use Lemma \ref
  {lem:averages-inlet} in the appendix.  We consider the following
  function $U_\eps$ on $Y\setminus \bar\Sigma_Y^\eps$ (we recall that
  $M_\eps = |I_\eps|$ denotes the number of considered cells)
  \begin{equation}
    \label{eq:U-eps}
    U_\eps(y) := \frac1{M_\eps} \sum_{k\in I_\eps} u^\eps(\eps (k+y))\,.
  \end{equation}
  Linearity of the Helmholtz equation implies that the rescaled
  function $U_\eps$ satisfies $-\Delta U_\eps = \omega^2 \eps^2
  U_\eps$ on $Y\setminus \Sigma_Y^\eps$. The (weak) two-scale
  convergence of $u^\eps(x)$ to $u_0(x,y)$ with $u_0(x,y) = w(x)$ for
  $y\in R_Y$ and $u_0(x,y) = v(x)$ for $y\in Q_Y$ implies the weak
  $L^2$-convergence of $U^\eps$ to constant functions,
  \begin{equation}
    \label{eq:U-convergence-23}
    U_\eps|_{R_Y} \weakto \mean_E w\qquad\text{and}\qquad
    U_\eps|_{Q_Y} \weakto \mean_E v\,.
  \end{equation}
  Since also the homogeneous Neumann boundary conditions on $\partial\Sigma_Y^\eps$ are
  satisfied, the sequence $U_\eps$ satisfies all assumptions of Lemma
  \ref {lem:averages-inlet} in the appendix.  Assertion \eqref
  {eq:result-lemma} of the lemma provides
  \begin{align}\label{eq:assertion-lemma-used}
    \mean_{\Gamma_Y^{\eps,R}} U_\eps(y)\,d\calH^{n-1}(y) \to \mean_E w
    \quad \text{ and }\quad
    \mean_{\Gamma_Y^{\eps,Q}} U_\eps(y)\,d\calH^{n-1}(y) \to \mean_E v\,.
  \end{align}

  It remains to relate inlet averages of $U_\eps$ (which live in the
  unit cell $Y$) with inlet averages of $u^\eps$ (which live on
  $\Omega$).  With the number of resonators in $E$ satisfying
  $M_\eps \eps^n \to |E|$ we calculate
  \begin{align*}
    \mean_{\Gamma_Y^{\eps,R}} U_\eps(y)\,d\calH^{n-1}(y) &= \mean_{\Gamma_Y^{\eps,R}}
    \frac1{M_\eps} \sum_{k\in I_\eps} u^\eps(\eps
    (k+y))\,d\calH^{n-1}(y)\\
    &= \frac{1}{A\eps^2} \frac1{M_\eps} \sum_{k\in I_\eps}
    \int_{\Gamma_Y^{\eps,R}} u^\eps(\eps (k+y))\,d\calH^{n-1}(y)\\
    &= \frac1{A\eps^2} \frac1{M_\eps} \frac1{\eps^{n-1}}
    \int_{\Gamma^\eps_R} u^\eps(x)\,d\calH^{n-1}(x)\\
    &=\left(\frac1{A |E| \eps} + \frac{o(1)}{\eps}\right)
    \int_{\Gamma^\eps_R} u^\eps(x)\,d\calH^{n-1}(x) \,.
  \end{align*}
  The convergence result \eqref{eq:assertion-lemma-used} thus implies
  \begin{align*}
    \frac1{A \eps} \int_{\Gamma^\eps_R} u^\eps(x)\,d\calH^{n-1}(x)
    \to \int_E w(x)\,dx\,,
  \end{align*}
  as $\eps\to 0$, which was the first claim in \eqref
  {eq:claim-geom-lim}. The second claim follows analogously from the
  second convergence in \eqref{eq:assertion-lemma-used}.
\end{proof}

\subsection{Effective equation}

In this section we derive the effective equation for the exterior
field $v$. We recall that we denoted interior and exterior volume as
$V=|R_Y|$ and $Q=|Q_Y|$.

\begin{proposition}[Effective equation for the exterior field]
  \label{prop:effectiveeq}
  In the situation of Theorem \ref{thm:main-2}, let $v$ be the
  exterior field and $w$ the interior field of Proposition
  \ref{prop:charactu0}. Then there holds
  \begin{align}
    \label{eq:effectiveeqvw}
    -\nabla\cdot\left(A_{*}\nabla v\right) =
    \Xi_{v,w}\quad\text{in }\Omega
  \end{align}
  in the sense of distributions with the effective coefficient matrix
  $A_{*}$, see \eqref{eq:effectiveA}. On the right hand side, we used
  the abbreviation
  \begin{align}
    \begin{split}
      \label{eq:Xivw}
      \Xi_{v,w}(x) :=
      \begin{cases}
        v(x)& \;\;\text{ for } x\in \Omega\setminus D, \\
        Qv(x)+Vw(x)& \;\; \text{ for } x\in D.
      \end{cases}
    \end{split}
  \end{align}
\end{proposition}

\begin{proof}
  Let $\theta\in C_c^\infty(\Omega)$ be arbitrary. The weak form of
  the Helmholtz equation \eqref{eq:Helmholtzeps} provides
  \begin{align}
    \label{eq:weakformHelmholtz}
    \int_{\Omega_\eps}\nabla
    u^\eps(x)\cdot\nabla\theta(x)\,dx
    =\omega^2\int_{\Omega_\eps}u^\eps(x)\theta(x)\,dx\,.
  \end{align}
  On the left hand side of \eqref{eq:weakformHelmholtz} we can
  directly pass to the limit $\eps\rightarrow 0$. Exploiting the
  characterization from Lemma \ref{lem:twoscalegradients} one obtains
  \begin{align*}
    &\int_{\Omega_\eps}\nabla u^\eps(x)\cdot\nabla_x\theta(x)\,dx\\
    &\stackrel{\eps\rightarrow 0}{\rightarrow}
    \int_\Omega\int_Y \xi_0(x,y)\cdot\nabla_x\theta(x)\,dy\,dx\\
    &=\int_D\int_{Q_Y}\left[\nabla_x v(x) + \nabla_y
      v_1(x,y)\right]\cdot\nabla_x\theta(x)\,dy\,dx
    +\int_{\Omega\setminus D}\nabla_x v(x)\cdot\nabla_x\theta(x)\,dx\\
    &=\int_D\int_{Q_Y} \left[\nabla_x v(x)+\sum_{i=1}^n\nabla_y
      \chi_i(y)\partial_{x_i}v(x)\right]\cdot\nabla_x\theta(x)\,dy\,dx\\
    &\qquad\qquad 
    +\int_{\Omega\setminus D}\nabla_x v(x)\cdot\nabla_x\theta(x)\,dx\\
    &= \int_D A_{\eff} \nabla_x v(x)\cdot\nabla_x \theta(x)\,dx
    +\int_{\Omega\setminus D}\nabla_x v(x)\cdot\nabla_x\theta(x)\,dx\\
    &=\int_\Omega A_{*}(x)\nabla_x v(x)\cdot\nabla_x \theta(x)\,dx\,,
  \end{align*}
  where in the fourth line we exploited the representation formula
  \eqref{eq:representv1} for $v_1$ and in the last two lines the
  definitions of $A_\eff$ and $A_{*}$ from \eqref{eq:effectiveA}.
  
  On the right hand side of \eqref{eq:weakformHelmholtz} we use the
  characterization from Proposition \ref{prop:charactu0} to find, in
  the limit $\eps\rightarrow 0$,
  \begin{align*}
    &\omega^2\int_{\Omega_\eps}u^\eps(x)\theta(x)\,dx\\
    &\rightarrow\, \omega^2\int_\Omega\int_Y u_0(x,y)\theta(x)\,dy\,dx\\
    &=\,\omega^2\int_D\left(\int_{Q_Y}v(x)\theta(x)\,dy +
      \int_{R_Y}w(x)\theta(x)\,dy\right)\,dx
    +\omega^2\int_{\Omega\setminus D}v(x)\theta(x)\,dx\\
    &=\,\omega^2\int_D\left(Qv(x)+Vw(x)\right)\theta(x)\,dx
    +\omega^2\int_{\Omega\setminus D}v(x)\theta(x)\,dx\\
    &=\,\omega^2\int_\Omega \Xi_{v,w}(x)\theta(x)\,dx\,.
  \end{align*}
  To sum up, we obtain
  \begin{align*}
    \int_\Omega A_{*}(x)\nabla_x v(x)\cdot\nabla_x
    \theta(x)\,dx=\omega^2\int_\Omega\Xi_{v,w}(x)\theta(x)\,dx\,.
  \end{align*}
  Since $\theta$ was arbitrary, this provides the claim \eqref
  {eq:effectiveeqvw}.
\end{proof}

\begin{proof}[Proof of Theorem \ref{thm:main-2}]
  Theorem \ref{thm:main-2} is a direct consequence of Proposition
  \ref{prop:effectiveeq}.  Indeed, by exploiting the Relation
  \eqref{eq:relationvw} between $v$ and $w$ one immediately obtains,
  for $x\in\Omega$,
  \begin{align*}
    \Xi_{v,w}(x)=\Lambda(x)v(x)\,.
  \end{align*}
  The effective equation \eqref{eq:effectiveeqvw} is therefore
  identical to the one of Theorem \ref{thm:main-2}.
\end{proof}

\appendix

\section{Averages on channel interfaces}
The proof of Proposition \ref{prop:geometricflow} is based on the
following auxiliary lemma. It shows that, in the limit $\eps\to 0$,
averages on interfaces coincide with bulk averages --- on the volume
side of the interface.

\begin{lemma}[Averages on thin channel
  interfaces]\label{lem:averages-inlet}
  We consider the obstacle $\Sigma_Y^\eps\subset Y$ that separates,
  inside the unit cell, the resonator $R_Y$ from the outer domain
  $Q_Y$, but leaves a thin channel $K_Y^\eps$ that connects $R_Y$ and
  $Q_Y$ (as introduced in Section \ref{sec.geometry}). Let $U_\eps :
  Y\setminus \Sigma_Y^\eps$ be a sequence of $H^1$-functions that is
  $L^2$-bounded and that solves the Helmholtz equation
  \begin{align}
    \label{eq:ass-lemma-1}
    -\Delta U_\eps &= \omega^2 \eps^2 U_\eps\quad\ \,
    \text{ in } Y\setminus \Sigma_Y^\eps\,,\\
    \del_n U_\eps &= 0\qquad\qquad \text{ on } \del\Sigma_Y^\eps\,.
    \label{eq:ass-lemma-2}
  \end{align}
  Note that we impose no boundary condition on $\del Y$. The only
  additional assumption is, for two numbers $\xi_R, \xi_Q\in \R$, the
  weak convergence
  \begin{align}
    \label{eq:ass-lemma-3}
    U_\eps|_{R_Y} &\weakto \xi_R\qquad \text{ in } L^2(R_Y)\,,\\
    U_\eps|_{Q_Y} &\weakto \xi_Q\qquad \text{ in } L^2(Q_Y)\,,
    \label{eq:ass-lemma-4}
  \end{align}
  as $\eps\to 0$.  We denote the interface between $R_Y$ and channel
  $K_Y^\eps$ by $\Gamma_Y^{\eps,R} := \bar R_Y \cap \bar K_Y^\eps =
  \{y_R\} \times \overline{B_{\alpha\eps^p}^{n-1}}$ and, accordingly, the outer
  interface by $\Gamma_Y^{\eps,Q} := \bar Q_Y \cap \bar K_Y^\eps =
  \{y_Q\} \times \overline{B_{\alpha\eps^p}^{n-1}}$. Under the above assumptions
  we obtain, in the limit $\eps\to 0$, the convergence of interface
  averages:
  \begin{align}\label{eq:result-lemma}
    \mean_{\Gamma_Y^{\eps,R}} U_\eps(y)\,d\calH^{n-1}(y) \to \xi_R
    \qquad \text{ and }\qquad
    \mean_{\Gamma_Y^{\eps,Q}} U_\eps(y)\,d\calH^{n-1}(y) \to \xi_Q\,.
  \end{align}
\end{lemma}

\begin{remark}
  It is easy to obtain local $H^1$-estimates for the solution sequence
  $U_\eps$ (as we will show in Step 1 of the proof). The assertion of
  the lemma is interesting since the channels degenerate in the limit
  $\eps\to 0$.  Indeed, let us imagine the channels had a constant
  cross-section which does not degenerate in the limit $\eps\to 0$. In
  that case, the trace theorem provided that the traces of $U_\eps$
  along $\Gamma_R^\eps$ converge to the trace of the limit function,
  i.e. the trace of the constant function $\xi_R$, which is
  $\xi_R$. The average of that function over $\Gamma_Y^{\eps,R}$ is
  $\xi_R$, hence \eqref {eq:result-lemma} follows.
\end{remark}

\begin{figure}[th]
   \centering
   \includegraphics[height=55mm]{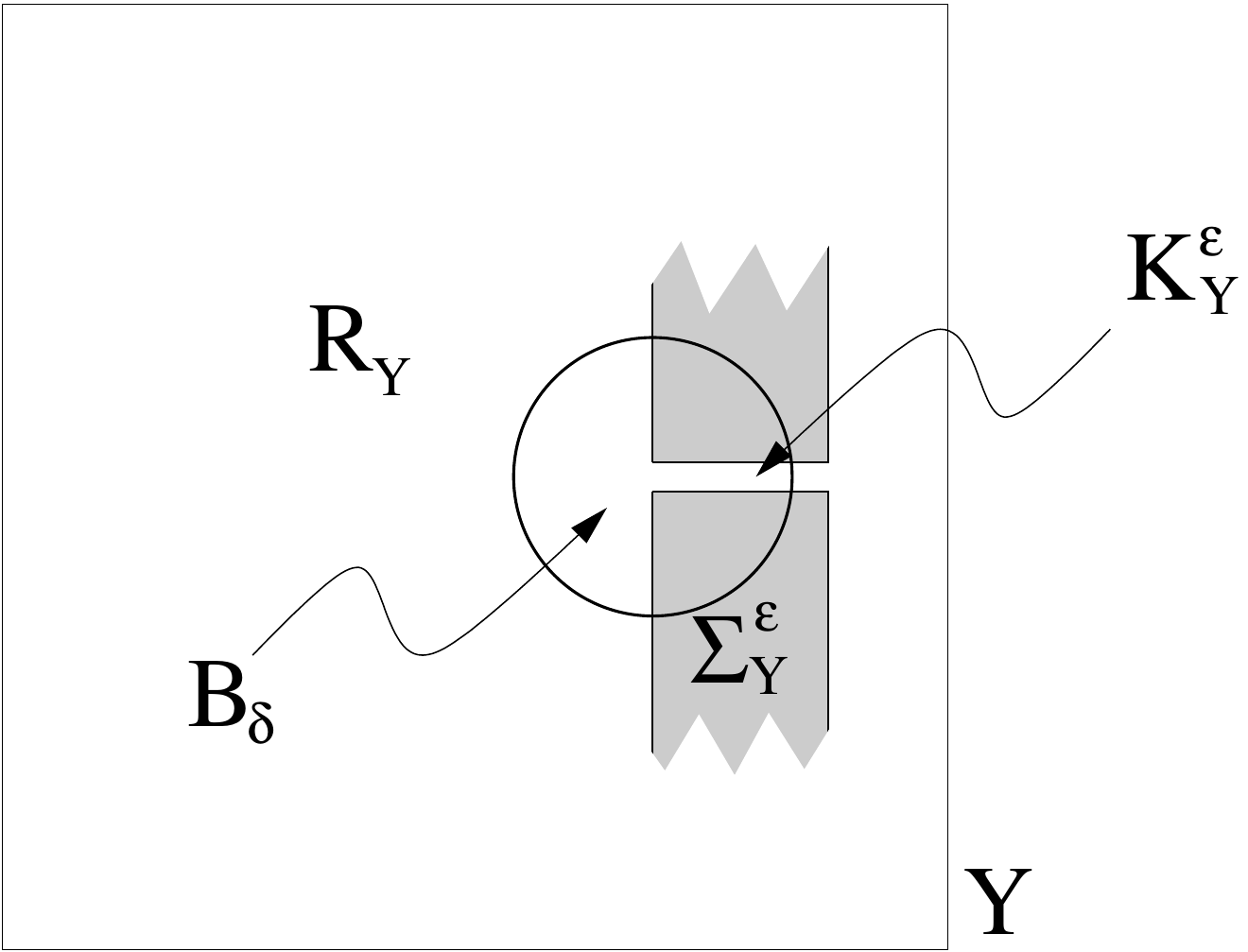}
   \caption{\em Sketch of the  geometry around the end-point of the channel.
     \label{fig:Lemma}}
\end{figure}

The subsequent proof is a sketch in so far as technical details in
Step 2 are omitted.

\begin{proof}[Sketch of proof]
  {\em Step 1. Local $H^1$ estimate.}  We choose a cut-off function
  $\theta\in C_c^\infty(Y)$ which is identical to $1$ in $R_Y\cup
  \Sigma_Y^\eps \cup K_Y^\eps$ and use $\theta^2 U_\eps$ as a test
  function in equation \eqref {eq:ass-lemma-1} to obtain
  \begin{align*}
    \int_{Y\setminus \Sigma_Y^\eps} |\nabla U_\eps|^2 \theta^2 = -
    2\int_{Y\setminus \Sigma_Y^\eps} \theta \nabla U_\eps\cdot\nabla
    \theta\, U_\eps + \omega^2 \eps^2 \int_{Y\setminus \Sigma_Y^\eps}
    |U_\eps|^2 \theta^2\,.
  \end{align*}
  With Youngs inequality we conclude from the $L^2$-boundedness of
  $U_\eps$ the $L^2$-boundedness of $\theta \nabla U_\eps$.

  \medskip {\em Step 2.  Estimates for $\del_1 U_\eps$.}  We now want
  to obtain estimates for higher derivatives of the solution
  sequence. We recall that the lateral boundaries of the channel are
  straight and aligned with $e_1$. Furthermore, in a neighborhood of
  $\Gamma_Y^{\eps,R}$, the boundary $\del R_Y$ was assumed to be
  contained in a hypersurface with normal $e_1$.

  We examine the solution in the ball $B_\delta := B_\delta((y_R,0))$,
  where $\delta>0$ is chosen so small that the boundary
  $\del\Sigma_Y^\eps \cap B_\delta$ consists of two pieces, a subset
  of the cylinder $\del K_Y^\eps$ and a subset of the plane $\{y_R\}
  \times \R^{n-1}$.

  We study the function $V_\eps := \del_1 U_\eps$ in $B_\delta$. The
  function $V_\eps$ solves the Helmholtz $-\Delta V_\eps = \omega^2
  \eps^2 V_\eps$ (in the distributional sense) in $B_\delta\setminus
  \Sigma_Y^\eps$. On the plane part of the boundary,
  $\del\Sigma_Y^\eps\cap \left(\{y_R\} \times \R^{n-1}\right) \cap
  B_\delta$, it satisfies the Dirichlet condition $V_\eps = 0$ (since
  $U_\eps$ satisfies the homogeneous Neumann condition). On the
  cylindrical part of the boundary, $\del K_Y^\eps\cap
  \del\Sigma_Y^\eps\cap B_\delta$, it satisfies the homogeneous
  Neumann condition ($\del_\nu V_\eps = \del_1 \del_\nu U_\eps = 0$).

  The homogeneous boundary conditions allow to derive an estimate for
  $V_\eps$: Multiplication of the Helmholtz equation for $V_\eps$ with
  the solution $V_\eps$ (multiplied with a cut-off function) provides
  local $H^1$-estimates for $V_\eps$.

  The above is not a rigorous proof: A priori, the function $V_\eps$
  is only of class $L^2$ (as a derivative of $U_\eps$), hence testing
  the equation with $V_\eps$ is not allowed. To obtain a rigorous
  proof, one has to proceed as follows: (a) Localize the problem and
  formulate boundary value problems for $U_\eps$ and $V_\eps$ in
  $B_\delta$ (or, better, for their truncated counterparts). (b) Show
  with the help of the Lax-Milgram lemma that the $V_\eps$ problem
  posesses a solution $V_\eps$ of class $H^1$. (c) Prove that the
  $y_1$-integrated solution solves the $U_\eps$-problem and conclude
  from the uniqueness of the $U_\eps$-problem that the $H^1$-function
  $V_\eps$ coincides with $\del_1 U_\eps$. We omit these technical
  details.

  \medskip {\em Step 3. Full high order estimate.}  In two space
  dimensions, we have obtained local $H^2$ estimates at this point:
  The second derivatives in the second direction can be expressed as
  $\del_2^2 U_\eps = -\del_1^2 U_\eps - \eps^2 \omega^2 U_\eps\in
  L^2$. We conclude that all second derivatives are bounded, $\|
  \Theta D^2 U_\eps \|_{L^2} \le C$ for a cut-off function $\Theta$.

  In space dimension $n=3$ we consider, for fixed $\zeta\in \R$,
  slices $S_{\zeta} := \{ y = (y_1,y_2,y_3) |\, y_1 = \zeta\}$. In every
  slice $S_{\zeta}$ the function $U_\eps|_{S_{\zeta}}$ solves the
  two-dimensional problem $\Delta_{2} U_\eps := (\del_2^2 +
  \del_3^2) U_\eps = -\del_1^2 U_\eps - \eps^2 \omega^2 U_\eps\in
  L^2(S_{\zeta})$ for almost every $\zeta$. This provides estimates
  for all second spatial derivatives in $L^2$, locally around the
  interface.

  \medskip {\em Step 4. Sobolev embedding and conclusion.} From the
  compact embedding $H^2(R_Y) \subset C^0(R_Y)$ we conclude that, up
  to a subsequence, the solution sequence $U_\eps|_{R_Y}$ converges
  not only weakly in $H^2(R_Y)$ to $\xi_R$, but also strongly in
  $C^0(R_Y)$. This implies the first claim of \eqref
  {eq:result-lemma}. The second claim is shown by exactly the same
  calculation for $\Gamma_Y^{\eps,Q}$.
\end{proof}

\vspace*{-2mm}
\subsection*{Acknowledgements}

Support of both authors by DFG grant Schw 639/6-1 is greatfully
acknowledged.

\bibliographystyle{abbrv} 
\bibliography{lit-manyhelmholtz-5}

\end{document}